\documentclass[11pt,reqno]{amsart}

\usepackage{amssymb, amsmath, amsfonts, latexsym}
\usepackage{cases}
\usepackage{bbm}
\usepackage{mathabx}
\usepackage{mathrsfs}
\usepackage{graphicx}
\usepackage{enumerate}
\usepackage{dsfont}
\usepackage{xcolor}

\newtheorem{thm}{Theorem}[section]
\newtheorem{corollary}[thm]{Corollary}
\newtheorem{lemma}[thm]{Lemma}
\newtheorem{proposition}[thm]{Proposition}

\newtheorem{con}[thm]{Condition}

\theoremstyle{definition}
\newtheorem{definition}{Definition}[section]
\newtheorem{defn}{Definition}[section]
\theoremstyle{remark}

\numberwithin{equation}{section}

\def\ZZ{\mathcal Z}
\def\e{\varepsilon}

\def\beq{\begin{equation}}
\def\nneq{\end{equation}}

\def\bdef{\begin{defn}}
\def\ndef{\end{defn}}

\def\bthm{\begin{thm}}
\def\nthm{\end{thm}}

\def\bprop{\begin{prop}}
\def\nprop{\end{prop}}

\def\brmk{\begin{remarks}}
\def\nrmk{\end{remarks}}

\def\bexa{\begin{exa}}
\def\nexa{\end{exa}}

\def\blem{\begin{lemma}}
\def\nlem{\end{lemma}}

\def\bcor{\begin{corollary}}
\def\ncor{\end{corollary}}

\def\bexe{\begin{exe}}
\def\nexe{\end{exe}}

\def\bprf{\begin{proof}}
\def\nprf{\end{proof}}

\def\bdes{\begin{description}}
\def\ndes{\end{description}}

\def\e{\varepsilon}
\def\<{\langle}
\def\>{\rangle}
\def\Re{{\mathrm{{\rm Re}}}}

\title[LDP for Stochastic generalized Ginzburg-Landau equation]{Large deviation principle for    stochastic generalized Ginzburg-Landau equation driven by  jump noise}

\author{Ran Wang}
\address[]{Ran Wang, School of Mathematics and Statistics,  Wuhan University,  Wuhan, 430072,
China.}
\email{rwang@whu.edu.cn}

\author{Beibei Zhang}
\address[]{Beibei Zhang, School of Mathematics and Statistics,  Wuhan University,  Wuhan, 430072,
China.}
\email{zhangbb@whu.edu.cn}

\date{}

\begin{document}
\maketitle

 \noindent {\bf Abstract:} In this paper, we establish a large deviation principle for the stochastic generalized Ginzburg-Landau equation driven by jump noise.  The main difficulties come from  the highly non-linear coefficient. Here we  adopt a new sufficient condition  for the weak convergence criteria, which is proposed by Matoussi, Sabbagh and Zhang (2021).

 \vskip0.3cm

 \vskip0.3cm

\vskip0.3cm

\section{Introduction} 
 Consider the  following stochastic generalized Ginzburg-Landau equation (SGGLE in short) driven by jump noise:
    \begin{eqnarray}\label{Eq SGGLE1}
    \begin{cases}
     du^{\e}(t)= \big((1+i\alpha)\Delta u^{\e}(t)-(1-i\beta)|u^{\e}(t)|^{2\sigma}u^{\e}(t)+\gamma u^{\e}(t)+F(u^{\e}(t))  \big)dt\\
     \ \ \ \ \ \ \ \ \ \ \
   +\e \int_{\ZZ}u^{\e}(t-)g(z)\tilde \eta^{\e^{-1}}(dz,dt), \\
    u^{\e}(x,t)=0,\ \ \ \ \ \ \ \  \ \ \ \ \    \ t\ge0, \, x\in \partial D,\\
    u^{\e}(x,0)=u(0),  \ \ \ \ \ \ \ \ x\in D,
    \end{cases}
\end{eqnarray}
where  $D=(0,L_1)\times (0,L_2), L_1,~L_2>0,~i=\sqrt{-1}, \gamma>0$, $\Delta$ is the Laplacian operator on $D$ with the Dirichlet boundary condition,  the parameters $\alpha, \beta, \gamma$ are   real-valued constants, $u$ is a complex-valued scalar function. The derivative term $F$ is defined by $F(u)=\lambda_1\cdot \nabla (|u|^2u)+(\lambda_2\cdot \nabla u)|u|^2$ with two complex constant vectors $\lambda_1$ and $\lambda_2$.   $\eta^{\e^{-1}}$ is a Poisson random measure on $[0,T]\times\mathcal{Z}$ with a $\sigma$-finite intensity measure $\e^{-1}\lambda_T\otimes\nu$, $\lambda_T$ is the Lebesgue measure on $[0,T]$ and $\nu$ is a $\sigma$-finite measure on $\mathcal{Z}$.   $\tilde \eta^{\e^{-1}}(ds,dz)= \eta^{\e^{-1}}(dsdz)-\e^{-1}\nu(dz)ds$
is a compensated Poisson random measure.

The generalized Ginzburg-Landau equation, as a basis model of the nonlinear phenomena, is used in various areas of the physics, such as  the Rayleigh-Benard convection, the Taylor-Couette flow in fluid mechanics,  the drift dissipative wave in plasma physics and the turbulent flow in chemical reaction
(see,  e.g., \cite{CaoGaowang1998, DuanT1992, GW2009, GuoGao1995, OL1996, LG2000}).
For the stochastic Ginzburg-Landau equation (SGLE in short), a long list of studies on the existence and uniqueness, as well as the asymptotic behavior of solutions, have been studied, such as \cite{Barton-Smith2004,LG, WG2008, Yang2004, Yang2010} and  references therein.
 In particular, Yang \cite{Yang2010} proved the existence and uniqueness of the solution  for the two-dimensional SGGLE with a multiplicative noise,  Lin and Gao \cite{LG} were concerned with the SGGLE driven by jumps.

The objective of this  paper is to study the large deviation principle for $u^{\e}$ given in Eq. \eqref{Eq SGGLE1}.
  Firstly, let us rewrite \eqref{Eq SGGLE1} into an abstract form.

 Let $(\Omega, \mathcal{F}, \{\mathcal{F}\}_{t\geq0}, \mathbb{P})$ be a filtered probability space. Denote $
\langle\cdot, \cdot\rangle$  the inner product  in $L^2(D)$ by
$
\langle u, v\rangle=\textmd{Re}\int_D u(x)\bar v(x)dx,
$
for $u, v\in H:= L^2(D) $, where $L^2(D)$ is the space of square integrable complex valued functions on $D$, $\bar v(x)$ is the adjoint of $v(x)$. We always write $H^1:=H^1(D)$, where $H^1(D)$ is the Sobolev space of the functions $f$ such that $f$ and $\nabla f$ are in $H$. Let
$$V:=H_0^1:=\{f\in H^1(D) \text{ such that } f \text{ vanishes on }\partial D\};$$
 $\|\cdot\|_p=\|\cdot\|_{L^p}$ and $\|\cdot\|=\|\cdot\|_2$ when $p=2$ for simplicity. $V^{*}$ is the dual space of $V$.

 Let
 \begin{align}\label{1-1}
 Au:=(1+i\alpha)\Delta u
 \end{align}
 and
 \begin{align}\label{1-2}
 Bu:=-(1-i\beta) |u|^{2\sigma}u+\gamma u+F(u).
 \end{align}
 The operator $A$ is an isomorphism from $D(A)=V\cap H^2$ onto $H$, where $H^2(D)$ is the Sobolev space of the functions $f$ such that $f$ , $\nabla f$ and $\Delta f$ are in $H$.
 We now rewrite \eqref{Eq SGGLE1} into the following abstract form:
    \begin{equation}\label{Eq SGGLE2}
    \begin{cases}
     du^{\e}(t)= \big(  Au^{\e}(t)+Bu^{\e}(t)  \big)dt    +\e \int_{\ZZ}u^{\e}(t-)g(z)\tilde \eta^{\e^{-1}}(dz,dt), \\
    u^{\e}(x,t)=0,\ \ \ \ \ \ \ \  \ \ \ \ \   \ \ \  \ \ \ \ \  \ \ \ \ \   t\ge0, x\in \partial D,\\
    u^{\e}(x,0)=u(0),  \ \ \ \ \ \ \ \  \ \ \  \ \ \ \ \ \ \ \ \ \  x\in D.
    \end{cases}
\end{equation}

\begin{definition} An $\mathcal{F}$-adapted process $u^{\e}=\big\{u^{\e}(t)\}_{t\in[0,T]}$ is said to be a solution of Eq. \eqref{Eq SGGLE2} if for its $dt\times d\mathbb P$-equivalent class $\tilde u^{\e}$, we have
\begin{itemize}
  \item[(1)] $\tilde u^{\e}\in L^2(\Omega; L^2([0,T];V))$;
    \item[(2)]  For every $t\in [0,T]$, 
    \begin{equation}\label{eq SPDE u-10}
    u^{\e}(t)=u(0)+\int_0^t A\tilde u^{\e}(s)ds+\int_0^t B\tilde u^{\e}(s) ds+\e\int_0^t \int_{\ZZ} \tilde u(s-)g(z) \tilde \eta^{\e^{-1}}(dz,ds)  \ \ \ \mathbb P{\text -a.s.}
    \end{equation}
 \end{itemize}
\end{definition}
\begin{con}$(\bf{Exponential\ Integrability})$\label{Condition Exponential Integrability}
There exists $\delta>0$ such that for all $E\in\mathcal{B}(\mathcal Z)$ such that
$\nu(E)<\infty$, the following holds
         $$
           \int_Ee^{\delta|g(z)|^2}\nu(dz)<\infty.
         $$
\end{con}

By using   the  argument  in the proof of Theorem 2.5 in Lin and Gao \cite{LG}, one can obtain   the following well-posedness result for Eq.\,\eqref{Eq SGGLE2}.
\begin{thm}\label{Sto-equ-solution}
Suppose Condition \ref{Condition Exponential Integrability} holds, $2\leq p<2\sigma$, $\sigma>2$ and $0<|\beta|<\frac{\sqrt{2\sigma+1}}{\sigma}$. If $u(0)\in V$ is deterministic, then there exists a unique $H$-valued progressively measurable process $u^\e$ such that $u^\e\in D([0,T],H)\cap L^2([0,T], V)\cap L^{2\sigma+2}([0,T], L^{2\sigma+2}(D))$   for any $T>0$ almost surely.
\end{thm}

The aim of  this   paper is to establish a large deviation principle for SGGLE \eqref{Eq SGGLE2}.
There are some references about
large deviation principles for the SGLEs,  see, e.g.,  Yang and Hou \cite{Yang2008} for    the SGLE with multiplicative Gaussian noise; Yang and Pu \cite{YP2015} for  the stochastic 3D cubic Ginzburg-Landau equation; Pu and Huang \cite{PH2019}   for the 2-D derivative Ginzburg-Landau equation.
   In their proofs,  the weak convergence method based on a variational representation for positive measurable functionals of  the  Brown motion is proved to be a powerful tool to establish the large deviation principles (see, e.g., \cite{Budhiraja-Dupuis2000, Budhiraja-Dupuis-M2008}).

The weak convergence criterion    for the case of Poisson random measures was   introduced by \cite{Budhiraja-Chen-Dupuis, Budhiraja-Dupuis-Maroulas.}. Recently, a sufficient condition to verify the large deviation criteria of Budhiraja,  Dupuis and Maroulas \cite{Budhiraja-Dupuis-M2008} is given    by Matoussi, Sabbagh and Zhang   \cite{AW2021}. It is also used by Liu, Song, Zhai and Zhang in \cite{LW2020} to study the LDP for the Mckean-Vlasov equation with jumps. The advantage of the new sufficient condition is to avoid proving the tightness of the controlled stochastic partial differential equation. This new sufficient condition is recently successfully applied to study the  large deviation principle in \cite{DWZ2020,  WZZ2021, WZ2020}.

 There are many results about the LDPs   related to stochastic partial differential equations driven by jump noise  (see, e.g., \cite{Budhiraja-Chen-Dupuis, Dong-Zhang, Rockner-Zhang, Swiech-Zabczyk, XiongZhai2018, Xu-Zhang, YZZ2015,Zhai-Zhang, Zhang-Zhou}), among which Xiong and Zhai \cite{XiongZhai2018} provided a unified proof  of LDPs for a large class of SPDEs with locally monotonic coefficients driven by L\'evy noise  by using the weak convergence approach. However, the method in \cite{XiongZhai2018} does not work for the SGGLE \eqref{Eq SGGLE2}, since the local monotonicity  Condition (i)-(iv) in \cite{XiongZhai2018}    cannot be satisfied  in a straight way. Although we can use the nonlinear structure and use the argument  in \cite{LG} to obtain the local monotonicity condition (ii), the coercivity (iii) and the growth condition (iv)  are very hard to verify for the SGGLE \eqref{Eq SGGLE2}. To overcome this diffculity, we use the argument of  \cite{LG} to get some preciser estimates in the study the LDP for  \eqref{Eq SGGLE2}.

 This paper is organized as follows.  In Section 2, we first recall the poisson random measure and the weak convergence criteria for the LDP  obtained in \cite{Budhiraja-Dupuis-M2008} and  \cite{AW2021}, and then we  present the main result of this paper.  Section 3 is devoted to study the skeleton equation. In Sections 4 and   5, we verify the two conditions for the  weak convergence criteria.

\section{Preliminaries and main results}
\subsection{Poisson random measure}\label{Section Representation}
Recall that $\ZZ$ is a locally compact Polish space. Denote by $\mathcal{M}_{FC}(\ZZ)$ the collection of all measures on $({\ZZ},\mathcal{B}({\ZZ}))$ such that $\nu(K)<\infty$ for any compact $K \in \mathcal{B}({\ZZ})$.
Denote by $C_c(\ZZ)$ the space of the continuous functions with compact supports, endowed $\mathcal{M}_{FC}({Z})$ with the weakest topology such that for every $f\in
C_c({\ZZ})$,  the function
\[\nu\longmapsto\langle f,\nu\rangle=\int_{{\ZZ}}f(u)d\nu(u),\]
is continuous for every $\nu\in\mathcal{M}_{FC}({\ZZ})$.
This topology can be metrized such that $\mathcal{M}_{FC}(\ZZ)$ is a Polish space (see, e.g.,
\cite{Budhiraja-Dupuis-Maroulas.}).

For any  $T\in(0,\infty)$, we denote $\ZZ_T=[0,T]\times \ZZ$ and $\nu_T=\lambda_T\otimes\nu$ with $\lambda_T$ being Lebesgue measure on $[0,T]$ and
$\nu\in\mathcal{M}_{FC}(\ZZ)$.
Let $\textbf{n}$ be a Poisson random measure on $\ZZ_T$ with intensity measure
$\nu_T$, it is well-known (\cite{Ikeda-Watanabe}) that $\textbf{n}$ is an $\mathcal{M}_{FC}(\ZZ_T)$ valued random variable such that

(i) for each
$B\in\mathcal{B}(\ZZ_T)$
with $\nu_T(B)<\infty$, $\textbf{n}(B)$ is Poisson distributed with mean $\nu_T(B)$;

(ii) for disjoint
$B_1,\cdots,B_k\in\mathcal{B}(\ZZ_T)$, $\textbf{n}(B_1),\cdots,\textbf{n}(B_k)$ are mutually independent
random
variables.

For notational simplicity, we write from now on
\begin{eqnarray}\label{eq P4 star}
\mathbb{M}=\mathcal{M}_{FC}(\ZZ_T),
\end{eqnarray}
and denote by $\mathbb{P}$ the probability measure induced by $\textbf{n}$ on $(\mathbb{M},\mathcal{B}(\mathbb{M}))$.
Under $\mathbb{P}$, the canonical map, $\eta:\mathbb{M}\rightarrow\mathbb{M},\ \eta(m)\doteq m$, is a Poisson random measure with
intensity measure $\nu_T$. With applications to large deviations in mind, we also consider, for $\theta>0$,
probability
measures $\mathbb{P}_\theta$ on $(\mathbb{M},\mathcal{B}(\mathbb{M}))$ under which $\eta$ is a Poisson random measure
with intensity $\theta\nu_T$. The corresponding expectation operators will be denoted by $\mathbb{E}$ and
$\mathbb{E}_\theta$,
respectively.

Denote
\begin{equation}\label{eq0302a}
\mathbb{Y}=\ZZ\times[0,\infty), \ \ \mathbb{Y}_T=[0,T]\times\mathbb{Y}, \ \ \bar{\mathbb{M}}=\mathcal{M}_{FC}(\mathbb{Y}_T).\end{equation}
Let $\bar{\mathbb{P}}$ be the unique probability measure on $(\bar{\mathbb{M}},\mathcal{B}(\bar{\mathbb{M}}))$
under which the canonical map, $\bar{\eta}:\bar{\mathbb{M}}\rightarrow\bar{\mathbb{M}},\bar{\eta}(\bar{m})\doteq \bar{m}$, is a Poisson
random
measure with intensity measure $\bar{\nu}_T=\lambda_T\otimes\nu\otimes \lambda_\infty$, with $\lambda_\infty$ being
Lebesgue measure on $[0,\infty)$.
The corresponding expectation operator will be denoted by $\bar{\mathbb{E}}$. Let
$\mathcal{F}_t\doteq\sigma\{\bar{\eta}((0,s]\times A):0\leq s\leq t,A\in\mathcal{B}(\mathbb{Y})\},$ and let
$\bar{\mathcal{F}}_t$
denote the completion under $\bar{\mathbb{P}}$. We denote by $\bar{\mathcal{P}}$ the predictable $\sigma$-field on
$[0,T]\times\bar{\mathbb{M}}$
with the filtration $\{\bar{\mathcal{F}}_t:0\leq t\leq T\}$ on $(\bar{\mathbb{M}},\mathcal{B}(\bar{\mathbb{M}}))$.
Let
$\bar{\mathcal{A}}$
be the class of all $(\bar{\mathcal{P}}\otimes\mathcal{B}(\mathcal{Z}))/\mathcal{B}([0,\infty))$-measurable maps
$\varphi:\mathcal{Z}_T\times\bar{\mathbb{M}}\rightarrow[0,\infty)$. For $\varphi\in\bar{\mathcal{A}}$, define a
counting process $\eta^\varphi$ on
$\ZZ_T$ by
   \begin{eqnarray}\label{Jump-representation}
      \eta^\varphi((0,t]\times U)=\int_{(0,t]\times U}\int_{(0,\infty)}1_{[0,\varphi(s,x)]}(r)\bar{\eta}(dsdxdr),\
      t\in[0,T],U\in\mathcal{B}(\mathcal{Z}).
   \end{eqnarray}
Here $\eta^\varphi$ is called a controlled random measure, with $\varphi$ selecting the intensity for the points at location
$x$
and time $s$, in a possibly random but non-anticipating way. When $\varphi(s,x,\bar{m})\equiv\theta\in(0,\infty)$, we
write $\eta^\varphi=\eta^\theta$. Note that $\eta^\theta$ has the same distribution with respect to $\bar{\mathbb{P}}$ as $\eta$
has with respect to $\mathbb{P}_\theta$.

\subsection{A general criterion for large deviation principle }

Let $\{u^\e\}_{\e>0} $ be a family of random variables defined on a probability space
$(\Omega,\mathcal{F},\mathbb{P})$
and taking values in a Polish space $\mathcal{E}$.
 \begin{definition}\label{Dfn-Rate function}
       \emph{\textbf{(Rate function)}} A function $I: \mathcal{E}\rightarrow[0,\infty]$ is called a rate function on
       $\mathcal{E}$,
       if for each $M<\infty$ the level set $\{y\in\mathcal{E}:I(y)\leq M\}$ is a compact subset of $\mathcal{E}$.

    \end{definition}
    \begin{definition}  \label{d:LDP}
       \emph{\textbf{(Large deviation principle)}} Let $I$ be a rate function on $\mathcal{E}$. The sequence
       $\{u^\e\}_{\e>0}$
       is said to satisfy a large deviation principle (LDP) on $\mathcal{E}$ with the rate function $I$ if the following two
       conditions
       hold:
\begin{itemize}
\item[(a).]   For each closed subset $F$ of $\mathcal{E}$,
              $$
                \limsup_{\e\rightarrow 0}\e \log\mathbb{P}(u^\e\in F)\leq\inf_{y\in F}I(y);
              $$

        \item[(b).]  For each open subset $G$ of $\mathcal{E}$,
              $$
                \liminf_{\e\rightarrow 0}\e \log\mathbb{P}(u^\e\in G)\geq-\inf_{y\in G}I(y).
              $$
              \end{itemize}
    \end{definition}
\vskip 3mm
Define $l:[0,\infty)\rightarrow[0,\infty)$ by
    $$
    l(r)=r\log r-r+1,\ \ r\in[0,\infty).
    $$
For any $\psi\in\bar{\mathcal{A}}$, the quantity
    \begin{eqnarray}\label{L_T}
      L_T(\varphi)=\int_{\ZZ_T}l(\psi(t,x,\omega))\nu_T(dtdx)
    \end{eqnarray}
is well defined as a $[0,\infty]$-valued random variable.
Define
   \begin{eqnarray}\label{S_N}
     S^N=\Big\{\varphi:\mathcal{Z}_T\rightarrow[0,\infty),\,L_T(\varphi)\leq N\Big\}.
   \end{eqnarray}
A function $\varphi\in S^N$ can be identified with a measure $\nu_T^\varphi\in\mathbb{M}$, defined by
   \begin{eqnarray*}
      \nu_T^\varphi(A)=\int_A \varphi(s,x)\nu_T(dsdx),\ \ A\in\mathcal{B}(\ZZ_T).
   \end{eqnarray*}
This identification induces a topology on $S^N$ under which $S^N$ is a compact space. Throughout this paper we use this topology on $S^N$. Denote
\begin{align}\label{S-definition}
\mathbb{S}:=\bigcup_{N=1}^\infty S^N
\end{align}
and $\bar{\mathbb{A}}^N:=\{\psi\in\bar{\mathcal{A}}\ \text{and}\ \psi(\omega)\in S^N,\ \bar{\mathbb{P}}\text{-}a.s.\}$. Let $\{K_n\subset \ZZ,\
n=1,2,\cdots\}$
be an increasing sequence of compact sets such that $\cup _{n=1}^\infty K_n=\ZZ$. For each $n$, let
   \begin{eqnarray*}
     \bar{\mathcal{A}}_{b,n}
          &=&
              \Big\{\psi\in\bar{\mathcal{A}}:
                                   {\rm \ for\ all\ }(t,\omega)\in[0,T]\times\bar{\mathbb{M}},\
                                   n\geq\psi(t,x,\omega)\geq
                                   1/n\ {\rm if}\ x\in K_n\ \\
                                  &&\ \ \ \ \ \ \ \ \ \ \ \ \
                                       {\rm and}\ \psi(t,x,\omega)=1\ {\rm if}\ x\in K_n^c
              \Big\},
   \end{eqnarray*}
and let $\bar{\mathcal{A}}_{b}=\cup _{n=1}^\infty\bar{\mathcal{A}}_{b,n}$. Define
$\tilde{\mathbb{A}}^N=\bar{\mathbb{A}}^N\cap\Big\{\phi: \phi\in\bar{\mathcal{A}}_b\Big\}$.
 
Let
$\{\mathcal{G}^\e \}_{\e >0}$
be a family of measurable maps from $\mathbb{M}$ to $\mathbb{U}$, where $\mathbb{M}$ is given in (\ref{eq P4 star}) and $\mathbb{U}$ is a Polish space. We present a sufficient condition established in \cite{LW2020, AW2021}
to obtain an LDP of the family
$\mathcal{G}^\e \Big(\e  \eta^{\e ^{-1}}\Big)$, as $\e \rightarrow 0$ .

\begin{con}\label{LDP}
 Suppose that there exists a measurable map $\mathcal{G}^0:\mathbb{M}\rightarrow \mathbb{U}$ such that the following hold:
\begin{itemize}
\item[(A).] For any $ N\in\mathbb{N}$, let $\varphi_n,\ \varphi\in S^N$ be such that $\varphi_n\rightarrow \varphi$ as
 $n\rightarrow\infty$. Then
      $$
         \mathcal{G}^0\left(\nu_T^{\varphi_n}\right)\rightarrow \mathcal{G}^0\left(\nu_T^{\varphi}\right)\quad\text{in}\quad \mathbb{U}.
      $$
\item[(B).] For any $N\in \mathbb{N}$, any $\varphi_\e \in\tilde{\mathbb{A}}^N$ and any $\delta>0$,
      $$
         \lim_{\e\rightarrow0}\mathbb P\left(d\left(\mathcal{G}^\epsilon\left(\e
         \eta^{\e ^{-1}\varphi_\e }\right),\mathcal{G}^0\left(\nu_T^{\varphi_\e}\right)\right)>\delta\right)=0.
      $$
 \end{itemize}
\end{con}
For $\phi\in\mathbb{U}$, define $\mathbb{S}_\phi=\Big\{\varphi\in S:\phi=\mathcal{G}^{0}\left(
\nu^\varphi_T\right)\Big\}$. Let $I:\mathbb{U}\rightarrow[0,\infty]$
be defined by
     \begin{eqnarray}\label{Rate function I}
        I(\phi)=\inf_{\varphi\in\mathbb{S}_\phi}L_T(\varphi),\ \qquad \phi\in\mathbb{U}.
     \end{eqnarray}
By convention, $I(\phi)=\infty$ if $\mathbb{S}_\phi=\emptyset$.
The following theorem was proved in Theorem 3.2 of \cite{AW2021} and Theorem 4.4 of \cite{LW2020}.
\begin{thm}\label{LDP-main-new}\cite{LW2020,AW2021}
For $\e>0$, let $u^\e$ be defined by $u^\e=\mathcal{G}^\e\Big(\e
\eta^{\e^{-1}}\Big)$, and suppose
that Condition \ref{LDP} holds. Then
the family $\{u^\e\}_{\e>0}$ satisfies a large deviation principle with the rate function $I$ defined by \eqref{Rate function I}.
\end{thm}

\subsection{Main result}
Recall that $u^{\e}$ is the solution to Eq. (\ref{Eq SGGLE2}) with a deterministic initial value $u(0) \in V$.
It follows from Theorem \ref{Sto-equ-solution} that, for every $\e>0$, there exists a measurable map $\mathcal{G}^{\e}:\bar{\mathbb{M}}\rightarrow D([0,T];H)$ such that for any Poisson random measure $\textbf{n}^{\e ^{-1}}$ on $[0,T]\times {\ZZ}$ with intensity measure $\e^{-1}\lambda_T\otimes\nu$ given on some probability space, $\mathcal{G}^{\e}\big(\e \textbf{n}^{\e ^{-1}}\big)$ is the unique solution of (\ref{Eq SGGLE2}) with $\tilde{\eta}^{\e ^{-1}}$ replaced by $\tilde{\textbf{n}}^{\e ^{-1}}$, here $\tilde{\textbf{n}}^{\e ^{-1}}$ is the compensated Poisson random measure of $\textbf{n}^{\e ^{-1}}$.

To state our main result, we need to introduce the map $\mathcal{G}^0$. Given $\varphi\in \mathbb{S}$ (Recall $\mathbb{S}$ defined in (\ref{S-definition})), we consider the following deterministic integral equation (the skeleton equation):
     \begin{eqnarray}\label{Eq  GGLE1}
         {u}^\varphi(t)
               =                 u(0)
             +   \int_0^tAu^\varphi(s)ds+\int_0^tBu^\varphi(s)ds
               +    \int_0^t\int_{\ZZ}{u}^\varphi(s)g(z)(\varphi(s,z)-1)\nu(dz)ds,
     \end{eqnarray}
where initial value $u(0) \in V$. By Proposition \ref{Main-thm-01} below, Eq. (\ref{Eq  GGLE1}) has a unique $u^{\varphi}\in C([0,T],H)\cap L^2([0,T],V)\cap L^{2\sigma+2}([0,T],L^{2\sigma+2}(D))$.
Recall that for $\varphi\in \mathbb{S}$,
$\nu_T^\varphi(ds,dz)=\varphi(s,z)\nu(dz)ds$.
Let
     \begin{eqnarray}\label{LDP-eq-01}
       \mathcal{G}^0(\nu^\varphi_T):=u^\varphi\ \text{for}\ \varphi\in \mathbb{S}.
     \end{eqnarray}
Then we can define the rate function $I:\mathbb{U}=D([0,T],H)\cap L^2([0,T], V)\cap L^{2\sigma+2}([0,T], L^{2\sigma+2}(D))$\\$\rightarrow [0,\infty]$   as in (\ref{Rate function I}).
The following is the main result of this paper.
\begin{thm}\label{Main-thm-02}
 Suppose that $u(0)\in V$, Condition \ref{Condition Exponential Integrability}
 holds and $2\leq p<2\sigma$, $\sigma>2$ and $0<|\beta|<\frac{\sqrt{2\sigma+1}}{\sigma}$. Then the family $\{u^\e \}_{\e >0}$ satisfies
 a large deviation principle on $D([0,T],H)\cap L^2([0,T], V)\cap L^{2\sigma+2}([0,T], L^{2\sigma+2}(D))$ with rate function $I$.
\end{thm}
\begin{proof}  According to Theorem \ref{LDP-main-new}, it is sufficient to prove that Condition \ref{LDP} is fulfilled. The verification of Condition \ref{LDP} (A) will be given by Proposition \ref{Prop-1}. Condition \ref{LDP} (B) will be established in Proposition \ref{Prop-2}. The proof of Theorem \ref{Main-thm-02} is complete.
\end{proof}
\section{The skeleton equation}
 In this section, we prove the existence and uniqueness of the solution to the skeleton equation (\ref{Eq  GGLE1}).
\subsection{Wellposedness    to the   skeleton equation}
\begin{proposition}\label{Main-thm-01}
    Suppose Condition \ref{Condition Exponential Integrability}
  holds, $\sigma>2$ and $0<|\beta|<\frac{\sqrt{2\sigma+1}}{\sigma}$. Then for any $u(0)\in V$ and for any $\varphi\in \mathbb{S}$,  (\ref{Eq  GGLE1}) admits  admits a unique solution $u^{\varphi}\in C([0,T],H)\cap L^2([0,T],V)\cap L^{2\sigma+2}([0,T],L^{2\sigma+2}(D))$.
Moreover, for any $N\in \mathbb{N}$, there exists $C_N>0$ such that
     \begin{eqnarray}\label{main-thm-01-ineq}
        \sup_{\varphi\in S^N}\Big(\sup_{s\in[0,T]}\|u^\varphi(s)\|^2+\int_0^T\|\nabla u^\varphi(s)\|^2ds+\int_0^T\|u^\varphi(s)\|^{2\sigma+2}_{2\sigma+2}ds\Big)
                \leq
        C_N.
     \end{eqnarray}
\end{proposition}

 We now give some lemmas for the proof of Proposition \ref{Main-thm-01}. The following two lemmas can be found in \cite{Budhiraja-Chen-Dupuis, YZZ2015}.

\begin{lemma}\label{3.1}
Assume  Condition (\ref{Condition Exponential Integrability}) holds. \\
\begin{itemize}
\item[(1).]\cite[Lemma 3.4]{Budhiraja-Chen-Dupuis} For  every $N\in\mathbb{N}$,
\begin{eqnarray}\label{g_H_1}
C^0_1:=\sup_{\varphi\in S^N}\int_{{\ZZ}_T}|g(z)|^2(\varphi(s,z)+1)\nu(dz)ds<\infty,
\end{eqnarray}

\begin{eqnarray}\label{g_H_2}
C^0_2:=\sup_{\varphi\in S^N}\int_{{\ZZ} _T}|g(z)|(\varphi(s,z)-1)\nu(dz)ds<\infty.
\end{eqnarray}

\item[(2).]\cite[Remark 2]{YZZ2015} For every $\eta>0$, there exists $\delta>0$ such that for any $A\subset[0,T]$ satisfying $\lambda_T(A)<\delta$,
\begin{eqnarray}\label{g_H_3}
\sup_{\varphi\in S^N}\int_A\int_{\ZZ}|g(z)||\varphi(s,z)-1|\nu(dz)ds\leq\eta.
\end{eqnarray}
\end{itemize}
\end{lemma}

\begin{lemma}\label{3.11}\cite[Lemma 3.11]{Budhiraja-Chen-Dupuis} For any $N\in \mathbb{N}$, let $\varphi_n,~\varphi\in S^N$ be such that $\varphi_n\rightarrow\varphi$ as $n\rightarrow\infty$.
Let $h:{\ZZ}\rightarrow\mathbb{R}$ be a measurable function such that
\begin{eqnarray*}
\int_{\ZZ}|h(z)|^2\nu(dz)<\infty
\end{eqnarray*}
and  for all $\delta\in(0,\infty)$
\begin{eqnarray*}
\int_{E}\exp(\delta|h(z)|)\nu(dz)<\infty,
\end{eqnarray*}
for all  $E\in \mathcal{B}({\ZZ})$ satisfying $\nu(E)<\infty$.
 Then
\begin{eqnarray*}
\lim_{n\rightarrow\infty}\int_{{\ZZ}_T}h(z)(\varphi_n(s,z)-1)\nu(dz)ds
=\int_{{\ZZ}_T}h(z)(\varphi(s,z)-1)\nu(dz)ds.
\end{eqnarray*}
\end{lemma}

Given $p>1$, $\delta\in (0,1)$, let $W^{\delta,p}([0,T];H)$ be the Sobolev space of all $u\in L^p([0,T];H)$ such that
$$
\int^T_0\int^T_0\frac{\|u(t)-u(s)\|^p}{|t-s|^{1+\delta p}}dtds<\infty,
$$
endowed with the norm

$$
\|u\|^p_{W^{\delta,p}([0,T];H)}=\int^T_0\|u(t)\|^pdt+\int^T_0\int^T_0\frac{\|u(t)-u(s)\|^p}{|t-s|^{1+\delta p}}dtds.
$$
By the criteria for compactness proved in Lions \cite[Section 5, Chapter I]{L1969}   and Temam \cite[Section 13.3]{T1983},   we have the following lemma.

\begin{lemma}\label{compact}
For any $p\in(1,\infty)$ and $\delta\in(0,1)$, let $\Lambda$ be the space
$$\Lambda=L^p([0,T];V)\cap W^{\delta,p}([0,T];V^*)$$
endowed with the natural norm. Then the embedding of $\Lambda$ in $L^{p}([0,T];H)$ is compact.
\end{lemma}

\subsection{Finite-dimensional approximations}

We use the Galerkin method to prove the existence of the solution to (\ref{Eq  GGLE1}).

Suppose that $\{e_i:i\in \mathbb{N}\}\subset D(A)$ is an orthonormal basis of $H$ such that ${\rm span}\{e_i:i\in \mathbb{N}\}$ is dense in $V$. Denote $H_n:={\rm span}\{e_1,...e_n\}$.
 Let $P_n$ be the orthogonal projection onto $H_n$ in $H$, i.e.,
 $$P_nx:=\sum\limits_{i=1}^{n}\langle x,e_i\rangle \,e_i.$$
 For simplicity, we denote
 \begin{eqnarray}\label{1-3}
 G(u):=A(u)+B(u),
 \end{eqnarray}
 where $A$ and $B$ are defined in (\ref{1-1}) and (\ref{1-2}), respectively. We consider the following equation:
\begin{eqnarray}\label{4}
u^\varphi_n(t)=u^\varphi_n(0)+\int_0^tP_nG(u^\varphi_n(s))ds+\int_0^t\int_{\ZZ} u^\varphi_n(s)g(z)(\varphi(s,z)-1)\nu(dz)ds,\ \
\end{eqnarray}
where $u^\varphi_n(0)=P_nu(0)$. By using the method of \cite{AlbeverioBWu2010}, we know that Eq. \eqref{4} has a unique solution.

 We first give some estimates which are useful to prove Theorem \ref{Main-thm-01}.
\begin{lemma}\label{Estimate-lemm-01}
Suppose Condition \ref{Condition Exponential Integrability}
  holds, $\sigma>2$ and $0<|\beta|<\frac{\sqrt{2\sigma+1}}{\sigma}$. Then for any $u(0)\in V$ and for any $\varphi\in \mathbb{S}$, there exists a constant $C_{3,1}\in(0,\infty)$ such that
\begin{eqnarray}
\sup\limits_{n\geq1}\left(\sup\limits_{0\leq s\leq T}\|u^\varphi_n(s)\|^2+\int_0^T\|\nabla u^\varphi_n(s)\|^2ds+\int_0^T\|u^\varphi_n(s)\|_{2\sigma+2}^{2\sigma+2}ds\right)\leq C_{3,1}(\|u(0)\|^2+1).\nonumber
\end{eqnarray}
\end{lemma}

\begin{proof}
Applying the chain rule to $\|u_n(t)\|^2$ and taking the real part, we obtain
\begin{eqnarray}\label{7}
\|u^\varphi_n(t)\|^2
&=&\|u^\varphi_n(0)\|^2+2\,\textmd{Re}\int_0^t \langle u^\varphi_n(s),(1+i\alpha)\Delta u^\varphi_n(s)\big>ds\nonumber\\
&&+2\int_0^t\big<u^\varphi_n(s),\gamma u^\varphi_n(s)\rangle ds
-2\,\textmd{Re}\int_0^t\langle u^\varphi_n(s),(1-i\beta)|u^\varphi_n(s)|^{2\sigma}u^\varphi_n(s)\rangle ds\nonumber\\
&&+2\,\textmd{Re}\int_0^t\langle u^\varphi_n(s),\lambda_1\cdot \nabla (|u^\varphi_n(s)|^2u^\varphi_n(s))+(\lambda_2\cdot \nabla u^\varphi_n(s))|u^\varphi_n(s)|^2\rangle ds\nonumber\\
&&+2\,\textmd{Re}
\int_0^t\int_\mathcal{Z}\langle u^\varphi_n(s),u^\varphi_n(s)g(z)(\varphi(s,z)-1)\rangle \nu(dz)ds\nonumber\\
&=&\|u^\varphi_n(0)\|^2-2\,\int_0^t\|\nabla u^\varphi_n(s)\|^2ds+2\,\gamma\int_0^t\|u^\varphi_n(s)\|^2ds\nonumber\\&&+2\,\textmd{Re}\int_0^t\langle u^\varphi_n(s),\lambda_1\cdot \nabla (|u^\varphi_n(s)|^2u^\varphi_n(s))+(\lambda_2\cdot \nabla u^\varphi_n(s))|u^\varphi_n(s)|^2\rangle ds\nonumber\\
&&-2\,\int_0^t\|u^\varphi_n(s)\|_{2\sigma+2}^{2\sigma+2}ds
+2\,\textmd{Re}
\int_0^t\int_\mathcal{Z}\langle u^\varphi_n(s),u^\varphi_n(s)g(z)(\varphi(s,z)-1)\rangle \nu(dz)ds\nonumber.
\end{eqnarray}
Then, by using the same technique in the proof of  (3.6) in \cite{LG}, there exists a constant $C_{3,2}=C_{3,2}(|\lambda_1|,|\lambda_2|)\in(0,\infty)$ such that
\begin{eqnarray*}
&&\|u^\varphi_n(t)\|^2+\frac{3}{2}\,\int_0^t\|\nabla u^\varphi_n(s)\|^2ds+\frac{3}{2}\,\int_0^t\|u^\varphi_n(s)\|_{2\sigma+2}^{2\sigma+2}ds\\
&\leq&\|u^\varphi_n(0)\|^2+\left(2\,\gamma+C_{3,2}(|\lambda_1|,|\lambda_2|)\right)\int_0^t\|u^\varphi_n(s)\|^2ds\\
&&+2\int_0^t\|u^\varphi_n(s)\|^2\int_\mathcal{Z}|g(z)||\varphi(s,z)-1|\nu(dz)ds.
\end{eqnarray*}
Applying Gronwall's lemma,   we have
\begin{eqnarray}
\nonumber
&&\sup\limits_{0\leq s\leq T}\|u^\varphi_n(s)\|^2+\int_0^T\|\nabla u^\varphi_n(s)\|^2ds+\int_0^T\|u^\varphi_n(s)\|_{2\sigma+2}^{2\sigma+2}ds\\
\nonumber
&\leq&\|u^\varphi_n(0)\|^2\cdot\exp\left[\int_0^T\left(\left(2\,\gamma+C_{3,2}(|\lambda_1|,|\lambda_2|)\right) +2\int_\mathcal{Z}|g(z)||\varphi(s,z)-1|\nu(dz)\right)ds\right]\\
&\leq&C_{3,1}(\|u(0)\|^2+1),\nonumber
\end{eqnarray}
for some constant $C_{3,1}\in(0,\infty)$.
Here we have used the fact of $\int_0^T\int_\mathcal{Z}|g(z)||\varphi(s,z)-1|\nu(dz)ds$ is finite by (\ref{g_H_2}). The proof is complete.
\end{proof}

\begin{lemma}\label{Estimate-lemm-03}
 Suppose Condition \ref{Condition Exponential Integrability}
  holds, $2\leq p<2\sigma$, $\sigma>2$ and $0<|\beta|<\frac{\sqrt{2\sigma+1}}{\sigma}$. Then for any $u(0)\in V$ and for any $\varphi\in \mathbb{S}$, there exists a constant $C_{3,3}\in (0,\infty)$ such that
\begin{eqnarray}
&&\sup\limits_{n\geq1}\bigg(\sup\limits_{0\leq s\leq t}\|\nabla u^\varphi_n(s)\|^p+\int_0^t\|\nabla u^\varphi_n(s)\|^{p-2}\|\triangle u^\varphi_n(s)\|^2ds\nonumber\\
&&\ \ \ \ + \int_0^t\int_D\|\nabla u^\varphi_n(s)\|^{p-2}|u^\varphi_n(s,x)|^{2\sigma}|\nabla u^\varphi_n(s,x)|^2dxds\bigg)\nonumber\\&\leq &C_{3,3}\,(\|\nabla u(0)\|^p+1).\nonumber
\end{eqnarray}
\end{lemma}

\begin{proof}
Applying the chain rule to $\|\nabla u^\varphi_n(t)\|^p$ and taking the real part, we obtain

\begin{eqnarray*}
\setlength{\abovedisplayskip}{3pt}
\setlength{\belowdisplayskip}{3pt}
\|\nabla u^\varphi_n(t)\|^p
&=&\|\nabla u^\varphi_n(0)\|^p+p\,\textmd{Re}\int_0^t\|\nabla u^\varphi_n(s)\|^{p-2}\langle\nabla u^\varphi_n(s),\nabla P_nG(u^\varphi_n(s))\rangle ds\\&&+p\,\textmd{Re}\int_0^t\int_\mathcal{Z}\|\nabla u^\varphi_n(s)\|^{p-2}\langle\nabla u^\varphi_n(s),\nabla u^\varphi_n(s)g(z)(\varphi(s,z)-1)\rangle\nu(dz)ds\nonumber\\
&\leq&\|\nabla u^\varphi_n(0)\|^p-p\,\textmd{Re}\int_0^t\|\nabla u^\varphi_n(s)\|^{p-2}\langle\Delta u^\varphi_n(s), P_nG(u^\varphi_n(s))\rangle ds\\&&+p\int_0^t\|\nabla u^\varphi_n(s)\|^{p}\int_\mathcal{Z}|g(z)||\varphi(s,z)-1|\nu(dz)ds.
\end{eqnarray*}

 According to (3.17) in \cite{LG}, there exists $\lambda_\beta>0$ such that
\begin{equation}\label{I_6}
\begin{split}
&-p\,\textmd{Re}\int_0^t|\nabla u^\varphi_n(s)\|^{p-2}\langle\Delta u^\varphi_n(s), P_nG(u^\varphi_n(s))\rangle ds\\
\leq&\,-\frac{p}{2}\int_0^t\|\nabla u^\varphi_n(s)\|^{p-2}\|\Delta u^\varphi_n(s)\|^2ds+
\left(p\gamma+C_{3,4}(\e_1, |\lambda_1|,|\lambda_2|)\right)\int_0^t\|\nabla u^\varphi_n(s)\|^pds\\
&+\left(-2\lambda_\beta+\e_1\right)\int_0^t\int_D\|\nabla u^\varphi_n(s)\|^{p-2}|u^\varphi_n(s,x)|^{2\sigma}|\nabla u^\varphi_n(s,x)|^2 dxds.
\end{split}
\end{equation}

Therefore, we have
\begin{eqnarray*}
&&\|\nabla u^\varphi_n(t)\|^p+\frac{p}{2}\int_0^t\|\nabla u^\varphi_n(s)\|^{p-2}\|\Delta u^\varphi_n(s)\|^2ds\\
&&+\left(2\lambda_\beta-\e_1\right)\int_0^t\int_D\|\nabla u^\varphi_n(s)\|^{p-2}|u^\varphi_n(s,x)|^{2\sigma}|\nabla u^\varphi_n(s,x)|^2 dxds\\
&\leq&\|\nabla u^\varphi_n(0)\|^p
+\left(p\gamma+C_{3,4}(\e_1, |\lambda_1|,|\lambda_2|)\right)\int_0^t\|\nabla u^\varphi_n(s)\|^pds\\
&&+p\int_0^t\|\nabla u^\varphi_n(t)\|^p\int_\mathcal{Z}|\varphi(s,z)-1||g(z)|\nu(dz)ds.
\end{eqnarray*}

Choosing sufficiently small $\e_1$ such that $2\lambda_\beta-\e_1>0$, by Gronwall's lemma and (\ref{g_H_2}) in Lemma \ref{3.1}, we get that there exists a constant $C_{3,3}\in(0,\infty)$, such that
\begin{eqnarray*}
\setlength{\abovedisplayskip}{3pt}
\setlength{\belowdisplayskip}{3pt}
&&\sup\limits_{0\leq s\leq t}\|\nabla u^\varphi_n(s)\|^p+\int_0^t\|\nabla u^\varphi_n(s)\|^{p-2}\|\triangle u^\varphi_n(s)\|^2ds\\
&&+ \int_0^t\int_D\|\nabla u^\varphi_n(s)\|^{p-2}|u^\varphi_n(s,x)|^{2\sigma}|\nabla u^\varphi_n(s,x)|^2dxds\\
&\leq&\|\nabla u^\varphi_n(0)\|^p\cdot\exp{\left[\int_0^T\left(\left(p\gamma+C_{3,4}(\e_1, |\lambda_1|,|\lambda_2|)\right)+p\int_\mathcal{Z}|g(z)||\varphi(s,z)-1|\nu(dz)\right)ds\right]}\\ &\leq&C_{3,3}(\|\nabla u(0)\|^p+1).
\end{eqnarray*}
The proof is complete.
\end{proof}

Next we study the convergence of approximating sequences.

\begin{lemma}\label{converge-lemm-01}
Suppose Condition \ref{Condition Exponential Integrability}
  holds, $\sigma>2$ and $0<|\beta|<\frac{\sqrt{2\sigma+1}}{\sigma}$. Let $T(u^\varphi_n)=-(1-i\beta)|u^\varphi_n|^{2\sigma}u^\varphi_n$, $q=\frac{2\sigma+2}{2\sigma+1}$ and $2\leq p<2\sigma$. Then for any $u(0)\in V$ and for any $\varphi\in \mathbb{S}$, there exist a subsequence of $\{u^\varphi_n\}$ (still denoted by $\{u^\varphi_n\}$),   $\tilde{u}^\varphi\in L^2((0,T];H^2)\cap L^\infty([0,T];V)\cap L^{2\sigma+2}([0,T];L^{2\sigma+2}(D))$, $\widetilde{T}\in L^q([0,T];L^q(D))$ and $\widetilde{F}\in L^2([0,T];H)$ such that
  \begin{itemize}
\item[(i)] $u^\varphi_n \rightarrow \tilde{u}^\varphi$\,\,weakly\,\,in\,\,$L^2((0,T];H^2)\cap L^{2\sigma+2}([0,T];L^{2\sigma+2}(D))\cap L^p([0,T];V)$;\\
\item[(ii)] $u^{\varphi_n}$ is weak-star converging to $\tilde{u}^\varphi$ in $L^\infty([0,T];V)$;\\
\item[(iii)] $P_n T(u^\varphi_n)\rightarrow \widetilde{T}$\,weakly\,\,in\,\,$L^q([0,T];L^q(D))$;\\
\item[(iv)]$P_n F(u^\varphi_n)\rightarrow \widetilde{F}$\,weakly\,\,in\,\,$L^2([0,T];H)$;\\
\item[(v)] $u^\varphi_n\rightarrow \tilde{u}^\varphi$ strongly\,\,in\,\,$L^2([0,T];V)$.\\
\end{itemize}
\end{lemma}
 \begin{proof}
 According to Lemma \ref{Estimate-lemm-01} and Lemma \ref{Estimate-lemm-03}, we obtain (i)-(iii).

 For (iv), since
\begin{align}\label{1-2-3-2-1}
\lambda_1\cdot \nabla (|u^\varphi_n|^2u^\varphi_n)+(\lambda_2\cdot \nabla u^\varphi_n)|u^\varphi_n|^2
=\left((2\lambda_1+\lambda_2)\cdot\nabla u^\varphi_n\right)|u^\varphi_n|^2+(\lambda_1\cdot\nabla u^\varphi_n)(u^\varphi_n)^2,
\end{align}
 by H\"older's inequality, Young's inequality and Lemma \ref{Estimate-lemm-03}, we have
 \begin{align*}
 \setlength{\abovedisplayskip}{3pt}
\setlength{\belowdisplayskip}{3pt}
 &\int^T_0\|P_n F(u^\varphi_n(s))\|^2ds\\
  \leq&\, 9(|\lambda_1|+|\lambda_2|)^2\int^T_0\int_D
|u^\varphi_n(s,x)|^4|\nabla u^\varphi_n(s,x)|^2dxds\\
  =&\,9(|\lambda_1|+|\lambda_2|)^2\int^T_0\int_D
|u^\varphi_n(s,x)|^4|\nabla u^\varphi_n(s,x)|^{\frac{4}{\sigma}}|\nabla u^\varphi_n(s,x)|^{2-\frac{4}{\sigma}}dxds\\
 \leq&\,9(|\lambda_1|+|\lambda_2|)^2\int^T_0\left(\int_D
|u^\varphi_n(s,x)|^{2\sigma}|\nabla u^\varphi_n(s,x)|^2dx\right)^{\frac{2}{\sigma}}\left(\int_D|\nabla u^\varphi_n(s,x)|^2dx\right)^{\frac{\sigma-2}{\sigma}}ds\\
 \leq&\,\e_2\int^T_0\int_D
|u^\varphi_n(s,x)|^{2\sigma}|\nabla u^\varphi_n(s,x)|^2dxds +C_{3,5}(\e_2, |\lambda_1|,|\lambda_2|)\int^T_0\|\nabla u^\varphi_n(s)\|^2ds\\
<&\,\infty.
 \end{align*}
This implies (iv).

For (v), we write $u^\varphi_n$ as
 \begin{equation}\label{Pn-1}
 \begin{split}
u^\varphi_n(t) =&\,u^\varphi_n(0)+ \int^t_0(1+i\alpha)\Delta u^\varphi_n(s)ds\\
 &+\int^t_0P_n\big(-(1-i\beta)|u^\varphi_n(s)|^{2\sigma}u^\varphi_n(s)+\gamma u^\varphi_n(s))\big)ds\\
 &+\int^t_0P_n\big(\lambda_1\cdot \nabla (|u^\varphi_n(s)|^2u^\varphi_n(s))+(\lambda_2\cdot \nabla u^\varphi_n(s))|u^\varphi_n(s)|^2\big)ds\\
 &+\int^t_0\int_\mathcal{Z}g(z)(\varphi(s,z)-1)u^\varphi_n(s)\nu(dz)ds\\
 =:&J_n^1+J_n^2+J_n^3+J_n^4+J_n^5,
\end{split}
\end{equation}
where $u^\varphi_n(0)=u(0)$.

Inspiring by the proof of Theorem 3.1 in \cite{FG1995} and Lemma 4.2 in \cite{Zhai-Zhang}, we show some estimates for $J_n^1$, $J_n^3$, $J_n^4$, $J_n^5$ and $J_n^2$ in order. By $u^\varphi_n(0)=u(0)$, we have  
\begin{eqnarray}\label{inequation-J_n}
\|J_n^1\|_V^2&\leq& L_1,
\end{eqnarray}
for the constant $L_1\in (0,\infty)$.

For $t>s$, by using H\"older's inequality, we obtain that there exists $C_{3,6}(\beta)\in(0,\infty)$ such that
\begin{eqnarray*}
&&\|J_n^3(t)-J_n^3(s)\|^2\\
&=&\big\|\int^t_sP_n\big(-(1-i\beta)|u^\varphi_n(l)|^{2\sigma}u^\varphi_n(l)+\gamma u^\varphi_n(l)  \big)dl\big\|^2\\
&\leq&\left(\int^t_s\|P_n\big(-(1-i\beta)|u^\varphi_n(l)|^{2\sigma}u^\varphi_n(l)\|+\gamma \| u^\varphi_n(l)\|dl\right)^2\\
&\leq&C_{3,6}(\beta) T\int^t_s\|u^\varphi_n(l)\|^{4\sigma+2}_{4\sigma+2}dl+\gamma^2 T\int^t_s\| u^\varphi_n(l)\|^2dl.
\end{eqnarray*}
Using the following Sobolev inequality in \cite[Theorem 12.83]{Leoni},
\begin{eqnarray}\label{Sobolev-inequality}
\|u^\varphi_n(t)\|_{4\sigma+2}\leq\|\nabla u^\varphi_n(t)\|^{\frac{\sigma}{2\sigma+1}}
\big\|u^\varphi_n(t)\big\|^{\frac{\sigma+1}{2\sigma+1}}_{2\sigma+2},
\end{eqnarray}
 we obtain by (\ref{Sobolev-inequality}) and Lemma \ref{Estimate-lemm-03},
\begin{eqnarray}\label{norm_est}
\int^t_s\|u^\varphi_n(l)\|^{4\sigma+2}_{4\sigma+2}dl
&\leq&\sup_{l\in[s,t]}\|\nabla u^\varphi_n(l)\|^{2\sigma}\int^t_s\big\|u^\varphi_n(l)\big\|^{2\sigma+2}_{2\sigma+2}dl\\
\nonumber
&\leq&C_{3,3}^{\frac{2\sigma}{p}}\left(\|\nabla u(0)\|^p+1\right)^{\frac{2\sigma}{p}}
\int^t_s\big\|u^\varphi_n(l)\big\|^{2\sigma+2}_{2\sigma+2}dl.
\end{eqnarray}
Therefore, by (\ref{norm_est}), we have
\begin{equation}\label{26}
 \|J_n^3(t)-J_n^3(s)\|^2
 \leq \, C_{3,7}\int^t_s\|u^\varphi_n(l)\|^{2\sigma+2}_{2\sigma+2}dl+\gamma^2 T\int^t_s\| u^\varphi_n(l)\|^2dl,
\end{equation}
where $C_{3,7}=C_{3,6}(\beta)TC_{3,3}^{\frac{2\sigma}{p}}\left(\|\nabla u(0)\|^p+1\right)^{\frac{2\sigma}{p}}$.
Applying the Fubini theorem, for any $\delta\in(0,\frac{1}{2})$, there exists a constant $L^1_3>0$ such that
 \begin{equation}\label{27}
 \begin{split}
 &\int^T_0\int^T_0\frac{\|J_n^3(t)-J_n^3(s)\|^2}{|t-s|^{1+2\delta}}dtds\\
 \leq&\, \int^T_0\int^T_0\left(C_{3,7}\int^t_s\|u^\varphi_n(t)\|^{2\sigma+2}_{2\sigma+2}dl+\gamma^2 T\int^t_s\| u^\varphi_n(l)\|^2dl\right)\cdot|t-s|^{-1-2\delta}dtds\\
   \leq&\, L^1_3\left(C_{3,7}\int^T_0\|u^\varphi_n(t)\|^{2\sigma+2}_{2\sigma+2}dl+\gamma^2 T\int^T_0\| u^\varphi_n(l)\|^2dl\right)
 \end{split}
\end{equation}
and
\begin{equation}\label{H-norm-J_3}
\begin{split}
\int^T_0\|J_n^3(t)\|^2dt
 \leq &\, C_{3,6}(\beta)\int^T_0\int^t_0\|u^\varphi_n(s)\|_{4\sigma+2}^{4\sigma+2}ds\cdot Tdt+\int^T_0\gamma^2\int^t_0\| u^\varphi_n(s)\|^2ds\cdot Tdt\\
 \leq &\, C_{3,7}T\int^t_0\|u^\varphi_n(s)\|^{2\sigma+2}_{2\sigma+2}ds+\gamma^2 T\int^t_0\| u^\varphi_n(s)\|^2ds.
\end{split}
\end{equation}
By (\ref{27}), (\ref{H-norm-J_3}) and Lemma \ref{Estimate-lemm-01}, for $L_3>0$, we have
\begin{eqnarray}\label{J_n^3}
\|J_n^3\|^2_{W^{\delta,2}([0,T];H)}&\leq& L_3.
\end{eqnarray}

Then we give some estimates of $J_n^4$.
By using (\ref{1-2-3-2-1}), the Cauchy-Schwarz inequality and the Young  inequality, we obtain
\begin{eqnarray*}
&&\|J_n^4(t)-J_n^4(s)\|^2\\
&=&\big\|\int^t_sP_n\big(\lambda_1\cdot \nabla (|u^\varphi_n(l)|^2u^\varphi_n(l))+(\lambda_2\cdot \nabla u^\varphi_n(l))|u^\varphi_n(l)|^2  \big)dl\big\|^2\\
&\leq&\left(\int^t_s3(|\lambda_1|+|\lambda_2|)\||u^\varphi_n(l)|^2|\nabla u^\varphi_n(l)|\|dl\right)^2\\
&\leq& T 9(|\lambda_1|+|\lambda_2|)^2\int^t_s\int_D|u^\varphi_n(l,x)|^4|\nabla u^\varphi_n(l,x)|^{\frac{4}{\sigma}}|\nabla u^\varphi_n(l,x)|^{2-\frac{4}{\sigma}}dxdl\\
&\leq& T 9(|\lambda_1|+|\lambda_2|)^2\int^t_s\left(\int_D|u^\varphi_n(l,x)|^{2\sigma}|\nabla u^\varphi_n(l,x)|^2dx\right)^{\frac{2}{\sigma}}\left(\int_D(|\nabla u^\varphi_n(l,x)|^2dx\right)^{\frac{\sigma-2}{\sigma}}dl\\
&\leq& \e_3\int^t_s\int_D|u^\varphi_n(l,x)|^{2\sigma}|\nabla u^\varphi_n(l,x)|^2dx+C_{3,8}(T,\e_3,|\lambda_1|,|\lambda_2|)\int^t_s\|\nabla u^\varphi_n(l,x)\|^2dl.
\end{eqnarray*}
For $J_n^4$, we give the similar estimate as $J_n^3$ in (\ref{27}). By using the Fubini theorem and Lemma \ref{Estimate-lemm-03}, we have that for any $\delta\in(0,\frac{1}{2})$, there exists a constant $L_4>0$ such that
\begin{eqnarray}\label{J_n^4}
\|J_n^4\|^2_{W^{\delta,2}([0,T];H)}&\leq& L_4.
\end{eqnarray}

For $J_n^5$, by (\ref{g_H_2}) in Lemma \ref{3.1}, we have
\begin{eqnarray*}
&&\|J_n^5(t)-J_n^5(s)\|^2\\
&=&\big\|\int^t_s\int_\mathcal{Z}P_n\left(u^\varphi_n(l)g(z)(\varphi(l,z)-1)\right)\nu(dz)dl\big\|^2\\
&\leq&\left(\int^t_s\int_\mathcal{Z}\|u^\varphi_n(l)\||g(z)||\varphi(l,z)-1|\nu(dz)dl\right)^2\\
&\leq&\sup_{l\in[0,T]}\|u^\varphi_n(l)\|^2\left(\int^t_s\int_\mathcal{Z}|g(z)||\varphi(l,z)-1|\nu(dz)dl\right)^2\\
&\leq&\sup_{l\in[0,T]}\|u^\varphi_n(l)\|^2\left(\int^T_0\int_\mathcal{Z}|g(z)||\varphi(l,z)-1|\nu(dz)dl\right)
\left(\int^t_s\int_\mathcal{Z}|g(z)||\varphi(l,z)-1|\nu(dz)dl\right)\\
&\leq&C^0_2\sup_{l\in[0,T]}\|u^\varphi_n(l)\|^2\left(\int^t_s\int_\mathcal{Z}|g(z)||\varphi(l,z)-1|\nu(dz)dl\right).
\end{eqnarray*}
For $J_n^5$, we give the similar estimate as $J_n^3$ in (\ref{27}). Applying the Fubini theorem, we have that for any $\delta\in(0,\frac{1}{2})$, there exists a constant $L_5>0$ such that
\begin{eqnarray}\label{J_n^5}
\|J_n^5\|^2_{W^{\delta,2}([0,T];H)}&\leq& L_5.
\end{eqnarray}
Since there exists a constant   $C_{3,9}(\alpha)\in(0,\infty)$ such that
\begin{eqnarray*}
\|J_n^2(t)-J_n^2(s)\|^2\leq C_{3,9}(\alpha) T\int^t_s\|\Delta u^\varphi_n(l)\|^2dl,
\end{eqnarray*}
which is finite by  using   Lemma \ref{converge-lemm-01}(1), by the same technique as $J_n^3$, we obtain that there exists $L_2>0$ such that
\begin{eqnarray}\label{inequation-J_n^2}
\|J_n^2\|^2_{W^{\delta,2}([0,T];H)}&\leq& L_2.
\end{eqnarray}

Combining (\ref{inequation-J_n}), (\ref{J_n^3}), (\ref{J_n^4}) and (\ref{inequation-J_n^2}), we have for $\delta\in(0,\frac{1}{2})$, there exists a constant $C_{3,10}(\delta)>0$ such that

\begin{eqnarray}\label{U_n}
\sup_{n\geq1}\left(\sup_{t\in[0,T]}\|u^\varphi_n\|^2_{W^{\delta,2}([0,T];H)}\right)\leq C_{3,10}(\delta),
\end{eqnarray}
 which is equivalent to
\begin{eqnarray}\label{U_n-new}
\sup_{n\geq1}\left(\sup_{t\in[0,T]}\|\nabla u^\varphi_n\|^2_{W^{\delta,2}([0,T];V^*)}\right)\leq C_{3,10}(\delta).
\end{eqnarray}
By (\ref{U_n-new}), Lemma \ref{compact} and Lemma \ref{Estimate-lemm-01}, we have that $\{\nabla u^\varphi_n\}_{n\ge1}$ is compact in $L^2([0,T];H)$. This implies that $\{u^\varphi_n\}_{n\ge1}$ is compact in $L^2([0,T];V)$. Therefore, we have $u^\varphi_n\rightarrow \tilde{u}^\varphi$ strongly in $L^2([0,T];V)$. The proof is complete.
\end{proof}

\subsection{Proof of Proposition \ref{Main-thm-01}}
In this part, we shall give the proof of Proposition \ref{Main-thm-01}, which is inspired by   Section 3 in \cite{LG}.
\begin{proof}[Proof of Proposition \ref{Main-thm-01}]
({\bf Existence}). Recall $\tilde{u}^\varphi$ and $\widetilde{T}$ are given in Lemma \ref{converge-lemm-01}.
Define
\begin{eqnarray}\label{28-X}
X^\varphi:=(1+i\alpha)\Delta\tilde{u}^\varphi+\widetilde{T}+\gamma \tilde{u}^\varphi+\widetilde{F}
\end{eqnarray}
and
\begin{eqnarray}\label{28-u}
u(t):=u(0)+\int_0^t X^\varphi(s)ds+\lim_{n\rightarrow\infty}\int_0^t\int_\mathcal{Z}u^\varphi_n(s)g(z)\left(\varphi(s,z)-1\right)\nu(dz)ds,
\end{eqnarray}
such that $u$ is a $V^{*}$-valued  modification of the $V$-valued process $\tilde{u}^\varphi$ in Lemma \ref{converge-lemm-01}, i.e., $u=\tilde{u}^\varphi$ in $V$. The aim of this part is to verify the following results:\\
\begin{equation}\label{X-Y-2}
\begin{split}
&\lim_{n\rightarrow\infty}\big\|\int_0^t\int_\mathcal{Z}u^\varphi_n(s)g(z)\left(\varphi(s,z)-1\right)\nu(dz)ds\\
&\ \ \ \ \ \ \ -\int_0^t\int_\mathcal{Z} \tilde{u}^\varphi(s)g(z)(\varphi(s,z)-1)\nu(dz)ds\big\|_V=0
\end{split}
\end{equation}
 and
\begin{eqnarray}\label{X-Y-1}
G\left(\tilde{u}^\varphi(s)\right)=X^\varphi(s),
\end{eqnarray}
where $G(u)=A(u)+B(u)$ is defined in (\ref{1-3}).

Notice that
(\ref{X-Y-2}) and (\ref{X-Y-1}) imply that $\tilde{u}^\varphi$ satisfies Eq.\,(\ref{Eq  GGLE1}). Thus, we obtain the existence of the solution to the skeleton equation (\ref{Eq  GGLE1}).

Let's give the proof of (\ref{X-Y-2}) and \eqref{X-Y-1}.
By Lemma \ref{Estimate-lemm-03} and Lemma \ref{converge-lemm-01}, it holds that
\begin{align*}
M:=\sup_{n\geq1}\sup_{t\in[0,T]}\|u^\varphi_n(t)\|_V\vee\sup_{t\in[0,T]}\|\tilde{u}^\varphi(t)\|_V<\infty.
\end{align*}
For every $\e_4>0$, define $$A_{n,\e_4}:=\{t\in[0,T]:\|u^\varphi_n(t)-\tilde{u}^\varphi(t)\|_V>\e_4\}.$$
By Lemma \ref{converge-lemm-01}(v), we have
\begin{eqnarray}\label{Y-3}
\lim_{n\rightarrow\infty}\lambda_T(A_{n,\e_4})\leq\lim_{n\rightarrow\infty}\frac{\int^T_0\|u^\varphi_n(t)-\tilde{u}^\varphi(t)\|_V^2dt}{\e_4^2}=0.
\end{eqnarray}
Then we obtain
\begin{equation}\label{Lim-normal}
\begin{split}
&\int^T_0\int_{\mathcal{Z}}\|\left(u^\varphi_n(s)-\tilde{u}^\varphi(s)\right)g(z)\left(\varphi(s,z)-1\right)\|_V\nu(dz)ds\\
\leq&\,
\int^T_0\int_{\mathcal{Z}}\|u^\varphi_n(s)-\tilde{u}^\varphi(s)\|_V|g(z)||\varphi(s,z)-1|\nu(dz)ds\\
 =&\,  \int_{A_{n,\e_4}}\int_{\mathcal{Z}}\|u^\varphi_n(s)-\tilde{u}^\varphi(s)\|_V|g(z)||\varphi(s,z)-1|\nu(dz)ds\\
& +\int_{A^c_{n,\e_4}}\int_{\mathcal{Z}}\|u^\varphi_n(s)-\tilde{u}^\varphi(s)\|_V|g(z)||\varphi(s,z)-1|\nu(dz)ds\\
 \leq &\,  2M\int_{A_{n,\e_4}}\int_{\mathcal{Z}}|g(z)||\varphi(s,z)-1|\nu(dz)ds\\
&+\e_4\int_{A^c_{n,\e_4}}\int_{\mathcal{Z}}|g(z)||\varphi(s,z)-1|\nu(dz)ds.
\end{split}
\end{equation}
By (\ref{g_H_2}) and (\ref{g_H_3}) in Lemma \ref{3.1}, we know that (\ref{X-Y-2}) holds.

Next we consider the proof of (\ref{X-Y-1}).

For any $\xi\in L^2((0,T];H^2)\cap L^{\infty}([0,T];V)\cap L^{2\sigma+2}([0,T];L^{2\sigma+2}(D))$, similarly as $(3.28)$ in \cite{LG},  we set
\begin{align}\label{r-set-definition}
\nonumber
r(t)&=\int^t_02\big[C_{3,11}(\e_5,\gamma)+\e_6\|\xi(s)\|^2+C_{3,12}(\e_7)\|\nabla \xi(s)\|^{\frac{2\sigma}{\sigma-1}}\\
\,\,\,\,\,\,\,\,\,\,\,\,\,\,&+C_{3,13}(\e_8)\|\nabla \xi(s)\|^{\frac{7\sigma-2}{\sigma+1}}+C_{3,14}(\e_9)\|\nabla \xi(s)\|^{\frac{10\sigma+4}{\sigma+4}}\big]ds,
\end{align}
where $2<\frac{7\sigma-2}{\sigma+1}<2\sigma$, $2<\frac{10\sigma+4}{\sigma+4}<2\sigma$  red and $\e_i\in(0,\infty)$, $i=5, \cdots, 9$. 

Applying the chain rule to $e^{-r(t)}\|u^\varphi_n(t)\|^2$ and taking the real part, we obtain
\begin{align*}
\setlength{\abovedisplayskip}{3pt}
\setlength{\belowdisplayskip}{3pt}
&\,\,\,\,\,e^{-r(t)}\|u^\varphi_n(t)\|^2-\|u^\varphi_n(0)\|^2\\
&=-\int^t_0e^{-r(s)}r'(s)\|u^\varphi_n(s)\|^2ds+2\textmd{Re}\int^t_0e^{-s}\langle u^\varphi_n(s),P_nG\left(u^\varphi_n(s)\right)\rangle ds\\
&\,\,\,\,\,+2\textmd{Re}\int^t_0\int_{\mathcal{Z}}e^{-r(s)}\langle u^\varphi_n(s),u^\varphi_n(s)g(z)\left(\varphi(s,z)-1\right)\rangle\nu(dz)ds\\
&=-\textmd{Re}\int^t_0e^{-r(s)}r'(s)\left(2\langle u^\varphi_n(s),\xi(s)\rangle-\|\xi(s)\|^2\right)ds\\
&\,\,\,\,\,+2\textmd{Re}\int^t_0e^{-r(s)}\left(\langle P_nG\left(u^\varphi_n(s)\right)-P_nG\left(\xi(s)\right),\xi(s)\rangle
+\langle P_nG\left(\xi(s)\right),u^\varphi_n(s)\rangle\right)ds\\
&\,\,\,\,\,+2\textmd{Re}\int^t_0\int_{\mathcal{Z}}e^{-r(s)}\langle u^\varphi_n(s),u^\varphi_n(s)g(z)\left(\varphi(s,z)-1\right)\rangle\nu(dz)ds\\
&\,\,\,\,\,-\int^t_0e^{-r(s)}r'(s)\|u^\varphi_n(s)-\xi(s)\|^2ds\\
&\,\,\,\,\,+2\Re\int^t_0e^{-r(s)}\langle P_nG\left(u^\varphi_n(s)\right)-P_nG\left(\xi(s)\right),u^\varphi_n(s)-\xi(s)\rangle ds.
\end{align*}

 Define $\bar{\e}_1=\e_5+\e_7+\e_8+\e_9$ and $\bar{\e}_2=\e_9+\e_{11}$. Choosing $\bar{\e}_1$, $\bar{\e}_2$ small enough such that $\bar{\e}_1<1$ and $K=-(1-\frac{\sigma}{\sqrt{2\sigma+1}}|\beta|)2^{-2\sigma}+\bar{\e}_2<0$, by Lemma 3.5 and Lemma 3.6 in \cite{LG}, we have
\begin{equation}\label{r-set}
\begin{split}
&-r'(s)\|u^\varphi_n(s)-\xi(s)\|^2+2\textmd{Re}\big<P_nG(u^\varphi_n(s))-P_nG(\xi(s)),u^\varphi_n(s)-\xi(s)\big>\\
 \leq &\, (-2+2\bar{\e}_1)\|\nabla (u^\varphi_n(s)-\xi(s))\|^2+2K\|\nabla (u^\varphi_n(s)-\xi(s))\|^{2\sigma+2}_{2\sigma+2}\\
 \leq &\, 0.
\end{split}
\end{equation}
Therefore, we have
\vspace{-0.2em}
\begin{align*}
&\,\,\,\,\,e^{-r(t)}\|u^\varphi_n(t)\|^2-\|u^\varphi_n(0)\|^2\\
&\leq-\textmd{Re}\int^t_0e^{-r(s)}r'(s)\left(2\langle u^\varphi_n(s),\xi(s)\rangle-\|\xi(s)\|^2\right)ds\\
&\,\,\,\,\,+2\textmd{Re}\int^t_0e^{-r(s)}\left(\langle P_nG\left(u^\varphi_n(s)\right)-P_nG\left(\xi(s)\right),\xi(s)\rangle
+\langle P_nG\left(\xi(s)\right),u^\varphi_n(s)\rangle\right)ds\\
&\,\,\,\,\,+2\textmd{Re}\int^t_0\int_{\mathcal{Z}}e^{-r(s)}\langle u^\varphi_n(s),u^\varphi_n(s)g(z)\left(\varphi(s,z)-1\right)\rangle\nu(dz)ds.
\end{align*}
By the lower semi-continuity of the weak convergence, we have
\begin{align}\label{311}
\nonumber
&\,\,\,\,\,e^{-T}\|u(T)\|^2-\|u(0)\|^2\\
\nonumber
&\leq\liminf_{n\rightarrow\infty}\left[e^{-r(T)}\|u^\varphi_n(T)\|^2-\|u^\varphi_n(0)\|^2\right]\\
\nonumber
&\leq\liminf_{n\rightarrow\infty}\left[-\textmd{Re}\int^T_0e^{-r(s)}r'(s)\left(2\langle u^\varphi_n(s),\xi(s)\rangle-\|\xi(s)\|^2\right)ds\right]\\
\nonumber
&\,\,\,\,\,+2\liminf_{n\rightarrow\infty}\left[\textmd{Re}\int^T_0e^{-r(s)}\left(\langle P_nG\left(u^\varphi_n(s)\right)-P_nG\left(\xi(s)\right),\xi(s)\rangle
+\langle P_nG\left(\xi(s)\right),u^\varphi_n(s)\rangle\right)ds\right]\\
\nonumber
&\,\,\,\,\,+2\liminf_{n\rightarrow\infty}\left[\textmd{Re}\int^T_0\int_{\mathcal{Z}}e^{-r(s)}\langle u^\varphi_n(s),u^\varphi_n(s)g(z)\left(\varphi(s,z)-1\right)\rangle\nu(dz)ds\right]\\
&=\textmd{Re}\int^T_0e^{-r(s)}r'(s)\left(\|\xi(s)\|^2-2\langle\tilde{u}^\varphi(s),\xi(s)\rangle\right)ds\\
\nonumber
&\,\,\,\,\,+2\textmd{Re}\int^T_0e^{-r(s)}\left(\langle X^\varphi(s)-G\left(\xi(s)\right),\xi(s)\rangle
+\langle G\left(\xi(s)\right),\tilde{u}^\varphi(s)\rangle\right)ds\\
\nonumber
&\,\,\,\,\,+2\liminf_{n\rightarrow\infty}\left[\textmd{Re}\int^T_0\int_{\mathcal{Z}}e^{-r(s)}\langle u^\varphi_n(s),u^\varphi_n(s)g(z)\left(\varphi(s,z)-1\right)\rangle\nu(dz)ds\right].
\end{align}

 On the other hand, applying the chain rule to $e^{-r(T)}\|u(T)\|^2$, taking the real part and then by using (\ref{28-u}), we obtain
  \begin{align}\label{312}
  \nonumber
&\,\,\,\,\,e^{-r(T)}\|u(T)\|^2-\|u(0)\|^2\\
&=-\int^T_0e^{-r(s)}r'(s)\|u(s)\|^2ds+2\textmd{Re}\int^T_0e^{-r(s)}\langle X^\varphi(s),\tilde{u}^\varphi(s)\rangle ds\\
\nonumber
&\,\,\,\,\,+2\lim_{n\rightarrow\infty}\left[\textmd{Re}\int^T_0\int_{\mathcal{Z}}e^{-r(s)}\langle\tilde{u}^\varphi(s),u^\varphi_n(s)g(z)\left(\varphi(s,z)-1\right)\rangle\nu(dz)ds\right].
\end{align}

We claim that
\begin{equation}\label{3133}
 \begin{split}
&\,\,\,\,\,\,\,\,\liminf_{n\rightarrow\infty}\left[\Re\int^T_0\int_{\mathcal{Z}}e^{-r(s)}\langle u^\varphi_n(s),u^\varphi_n(s)g(z)\left(\varphi(s,z)-1\right)\rangle\nu(dz)ds\right]
\\
&=\lim_{n\rightarrow\infty}\left[\textmd{Re}\int^T_0\int_{\mathcal{Z}}e^{-r(s)}\langle\tilde{u}^\varphi(s),u^\varphi_n(s)g(z)\left(\varphi(s,z)-1\right)\rangle\nu(dz)ds\right].
\end{split}
\end{equation}
The proof of (\ref{3133})  is postponed below. We first use  (\ref{3133}) to prove the existence of solution to Eq.\,\eqref{Eq  GGLE1}.
 Using    (\ref{311}),  (\ref{312}) and (\ref{3133}), we obtain
 \begin{eqnarray*}
-\int^T_0e^{-r(s)}r'(s)\|\tilde{u}^\varphi(s)-\xi(s)\|^2ds+2\textmd{Re}\int^T_0e^{-r(s)}\langle X^\varphi(s)-G\left(\xi(s)\right),\tilde{u}^\varphi(s)-\xi(s)\rangle ds\leq0.
\end{eqnarray*}
Taking $\xi=\tilde{u}^\varphi-\e_{10} v$, $v\in L^\infty([0,T],V),\,\,\e_{10}>0$,
we arrive at
\begin{eqnarray*}
-\int^T_0e^{-r(s)}r'(s)\e_{10}^2\|v(s)\|^2ds+2\e_{10}\textmd{Re}\int^T_0e^{-r(s)}\langle X^\varphi(s)-G\left(\tilde{u}^\varphi(s)-\e_{10} v(s)\right),v(s)\rangle ds\leq0.
\end{eqnarray*}
Hence
\begin{eqnarray*}
-\int^T_0e^{-r(s)}r'(s)\e_{10}\|v(s)\|^2ds+2\textmd{Re}\int^T_0e^{-r(s)}\langle X^\varphi(s)-G\left(\tilde{u}^\varphi(s)-\e_{10}v(s)\right),v(s)\rangle ds\leq0.
\end{eqnarray*}
Since $v$ is arbitrary, letting $\e_{10}\rightarrow0$, we get $X^\varphi(s)=G\left(\tilde{u}^\varphi(s)\right)$.

We now prove (\ref{3133}). A straightforward calculation gives
\begin{equation}\label{normal_0}
\begin{split}
&\,\,\,\,\,\,\,\,\textmd{Re}\int^T_0\int_{\mathcal{Z}}e^{-r(s)}\langle \left(u^\varphi_n(s)-\tilde{u}^\varphi(s)\right),u^\varphi_n(s)g(z)\left(\varphi(s,z)-1\right)\rangle\nu(dz)ds\\
 \leq&\,
\int^T_0\int_{\mathcal{Z}}|e^{-r(s)}|\|u^\varphi_n(s)-\tilde{u}^\varphi(s)\||g(z)||\varphi(s,z)-1|\|u^\varphi_n(s)\|\nu(dz)ds\\
 \leq &\,  M\int^T_0\int_{\mathcal{Z}}\|u^\varphi_n(s)-\tilde{u}^\varphi(s)\||g(z)||\varphi(s,z)-1|\nu(dz)ds.
\end{split}
\end{equation}
This, together with   (\ref{Lim-normal}),  implies  (\ref{3133}). Namely, we finish the proof of (\ref{X-Y-1}).

({\bf Uniqueness}). Suppose  that $u^\varphi$ and $v^\varphi$ are two solutions of Eq.\,\eqref{Eq  GGLE1} with the same initial condition $u(0)$. Applying the chain rule to $e^{-r(t)}\|u^\varphi(t)-v^\varphi(t)\|^2$ and taking the real part, we have
\begin{align*}
&\,\,\,\,\,e^{-r(t)}\|u^\varphi(t)-v^\varphi(t)\|^2\\
&=-\int^t_0e^{-r(s)}r'(s)\|u^\varphi(s)-v^\varphi(s)\|^2ds\\
&\,\,\,\,\,+2\textmd{Re}\int^t_0e^{-r(s)}\langle u^\varphi(s)-v^\varphi(s), G(u^\varphi(s))-G(v^\varphi(s))\rangle ds\\
&\,\,\,\,\,+2\textmd{Re}\int^t_0\int_{\mathcal{Z}}e^{-r(s)}\langle u^\varphi(s)-v^\varphi(s),\left(u^\varphi(s)-v^\varphi(s)\right)g(z)\left(\varphi(s,z)-1\right)\rangle\nu(dz)ds.
\end{align*}
It follows from (\ref{r-set}) that
\begin{align*}
e^{-r(t)}\|u^\varphi(t)-v^\varphi(t)\|^2\leq2\textmd{Re}\int^t_0\int_{\mathcal{Z}}e^{-r(s)}\|u^\varphi(s)-v^\varphi(s)\|^2|g(z)||\varphi(s,z)-1|\nu(dz)ds.
\end{align*}
By Gronwall's inequality, we obtain the uniqueness of the solution and complete the proof of Proposition \ref{Main-thm-01}.
\end{proof}

\section{Verification of Condition \ref{LDP} (a) }
 We now verify  Condition \ref{LDP} (a). Recall that $\mathcal{G}^0(\nu_T^{\varphi})=u^\varphi$ for $\varphi\in \mathbb{S}$, where $u^\varphi$ satisfies Eq.\,\eqref{Eq  GGLE1}  and $\mathbb{S}$ is defined in (\ref{S-definition}).

\begin{proposition}\label{Prop-1}
 For any  $N\in\mathbb{N}$,   let $\varphi_n,\, \varphi\in S^N$ be such that $\varphi_n\rightarrow \varphi$ as $n\rightarrow\infty$. Then
     \begin{eqnarray*}
       \mathcal{G}^0(\nu_T^{\varphi_n})\rightarrow\mathcal{G}^0(\nu_T^{\varphi})\quad \text{in}\quad C([0,T],H)\cap L^2([0,T], V)\cap L^{2\sigma+2}([0,T], L^{2\sigma+2}(D)).
     \end{eqnarray*}
\end{proposition}

\begin{proof}
Using the same arguments in Subsection 3.2 and combining with Lemma \ref{3.1}, we can deduce that for any $\delta\in(0,\frac{1}{2})$ and $2\leq p<2\sigma$, there exist some positive constants $C_{4,1}(N, \|u(0)\|)$, $C_{4,2}(N, \|\nabla u(0)\|)$ and $C_{4,3}(N, \delta)$ such that
\begin{itemize}
\item[(a)]
\begin{equation}\label{4444}
\begin{split}
&\sup\limits_{0\leq s\leq T}\|u^{\varphi_n}(s)\|^2+\int_0^T\|\nabla u^{\varphi_n}(s)\|^2ds+\int_0^T\|u^{\varphi_n}(s)\|_{2\sigma+2}^{2\sigma+2}ds\\
\leq &\, C_{4,1}(N, \|u(0)\|);
\end{split}
\end{equation}
\item[(b)]
\begin{equation}\label{u^n-est}
\begin{split}
 &\sup\limits_{0\leq s\leq t}\|\nabla u^{\varphi_n}(s)\|^p+\int_0^t\|\nabla u^{\varphi_n}(s)\|^{p-2}\|\triangle u^{\varphi_n}(s)\|^2ds\\
&+ \int_0^t\int_D\|\nabla u^{\varphi_n}(s)\|^{p-2}|u^{\varphi_n}(s,x)|^{2\sigma}|\nabla u^{\varphi_n}(s,x)|^2dxds\\
\leq&\,  C_{4,2}(N, \|\nabla u(0)\|);
\end{split}
\end{equation}
\item[(c)]
\begin{align}
\|u^{\varphi_n}\|^2_{W^{\delta,2}([0,T];H)}\leq C_{4,3}(N, \delta).
\end{align}
\end{itemize}
By Lemma \ref{converge-lemm-01}, we can assert that there exist processes $\tilde{u}_1^\varphi\in L^2((0,T];H^2)\cap L^\infty([0,T];V)\cap L^{2\sigma+2}([0,T];L^{2\sigma+2}(D))$, $\widetilde{T}_1\in L^q([0,T];L^q(D))$ and $\widetilde{F}_1\in L^2([0,T];H)$ such that, as $n\rightarrow\infty$,\\
 \begin{itemize}
\item[(i)] $u^{\varphi_n} \rightarrow \tilde{u}_1^\varphi$\,\,weakly\,\,in\,\,$L^2((0,T];H^2)\cap L^{2\sigma+2}([0,T];L^{2\sigma+2}(D))\cap L^p([0,T];V)$;\\
 \item[(ii)] $u^{\varphi_n}$ is weak-star converging to $\tilde{u}_1^\varphi$ in $L^\infty([0,T];V)$;\\
 \item[(iii)] $T(u^{\varphi_n})\rightarrow \widetilde{T}_1$\,\,weakly\,\,in\,\,$L^q([0,T];L^q(D))$;\\
 \item[(iv)] $P_n F(u^{\varphi_n})\rightarrow \widetilde{F}_1$\,weakly\,\,in\,\,$L^2([0,T];H)$;\\
\item[(v)] $u^{\varphi_n}\rightarrow \tilde{u}_1^\varphi$\,\,strongly\,\,in\,\,$L^2([0,T];V)$.\\
\end{itemize}
Our aim is  to show that $\tilde{u}_1^\varphi=u^\varphi$.
To do this,  it suffices to prove that
\begin{equation}\label{eq claim}
\begin{split}
\tilde{u}_1^\varphi(t)=&\,u(0)+ \int^t_0(1+i\alpha)\Delta \tilde{u}_1^\varphi(s)ds+\int^t_0\big(-(1-i\beta)|\tilde{u}_1^\varphi(s)|^{2\sigma}\tilde{u}_1^\varphi(s)+\gamma \tilde{u}_1^\varphi(s))  \big)ds\\
 &+\int^t_0\lambda_1\cdot \nabla (|\tilde{u}_1^\varphi(s)|^2\tilde{u}_1^\varphi(s))+(\lambda_2\cdot \nabla \tilde{u}_1^\varphi(s))|\tilde{u}_1^\varphi(s)|^2ds\\
 &+\int^T_0\int_\mathcal{Z}\tilde{u}_1^\varphi(s)g(z)(\varphi(s,z)-1)\nu(dz)ds.
\end{split}
\end{equation}
Similarly to   the proof of Proposition \ref{Main-thm-01}, to prove \eqref{eq claim}, it suffices to prove the following two identities:
  \begin{itemize}
  \item[(a)]\begin{equation}\label{X-Y-4}
 \begin{split}
& \lim_{n\rightarrow\infty}\Big\|\int_0^t\int_\mathcal{Z}u^{\varphi_n}(s)g(z)\left(\varphi_n(s,z)-1\right)\nu(dz)ds\\
 &\ \ \ \ \ \ \ -\int_0^t\int_\mathcal{Z} \tilde{u}_1^\varphi(s)g(z)(\varphi(s,z)-1)\nu(dz)ds\Big\|_V=0.
\end{split}
\end{equation}
     \item[(b)] For the process $Y^\varphi$   defined by
 \begin{eqnarray*}
Y^\varphi:=(1+i\alpha)\Delta\tilde{u}_1^\varphi+\widetilde{T}_1+\gamma \tilde{u}_1^\varphi,
\end{eqnarray*}
we have
\begin{eqnarray}\label{X-Y-3}
G(\tilde{u}_1^\varphi)=Y^\varphi.
\end{eqnarray}
      \end{itemize}

 Next, we prove \eqref{X-Y-4} and \eqref{X-Y-3} separately.  
Set
\begin{align*}
\widetilde{I^1_n}&:=\int_0^t\int_\mathcal{Z}u^{\varphi_n}(s)g(z)\left(\varphi_n(s,z)-1\right)\nu(dz)ds,\\
\widetilde{I^2_n}&:=\int_0^t\int_\mathcal{Z}u^{\varphi_n}(s)g(z)\left(\varphi(s,z)-1\right)\nu(dz)ds,\\
\widetilde{I}&:=\int_0^t\int_\mathcal{Z} \tilde{u}_1^\varphi(s)g(z)(\varphi(s,z)-1)\nu(dz)ds.
\end{align*}

By (\ref{X-Y-2}), we have
\begin{align}\label{I_n_2-I}
\lim_{n\rightarrow\infty}\|\widetilde{I^2_n}-\widetilde{I}\|_V=0.
\end{align}

Noticing that  $u^{\varphi_n} \in L^\infty([0,T];V)$, by Lemma \ref{3.11} and (\ref{4444}), we have
\begin{equation}\label{I_n_2-I_n_1}
\begin{split}
 &\|\widetilde{I^1_n}-\widetilde{I^2_n}\|_V\\
 \leq&\,\sup_{0\leq t\leq T}\|u^{\varphi_n}(t)\|_V\cdot
\bigg|\int_0^T\int_\mathcal{Z} \left((\varphi_{n}(t,z)-1)-(\varphi(t,z)-1)\right)|g(z)|\nu(dz)dt\bigg|\\
&\, \longrightarrow  0.
\end{split}
\end{equation}
Combining (\ref{I_n_2-I}) and (\ref{I_n_2-I_n_1}), we obtain (\ref{X-Y-4}).

Recall $r$ defined in (\ref{r-set-definition}). By the proof of (\ref{X-Y-1}), to prove \eqref{X-Y-3}, it suffices to  to prove
\begin{equation}\label{3344}
 \begin{split}
&\lim_{n\rightarrow\infty}\left[\int_0^T\int_\mathcal{Z} e^{-r(s)}\big< u^{\varphi_n}(s)g(z)(\varphi_n(s,z)-1),u^{\varphi_n}(s)\big>\nu(dz)ds\right]\\
&\, \,\,  \,\,\,\,\,-\int_0^T\int_\mathcal{Z} e^{-r(s)}\big< u^{\varphi_n}(s)g(z)(\varphi_n(s,z)-1),\tilde{u}_1^\varphi(s)\big>\nu(dz)ds=0.
\end{split}
\end{equation}
A straightforward calculation gives
\begin{align*}
&\,\,\,\,\,\,\bigg|\int_0^T\int_\mathcal{Z} e^{-r(s)}\langle u^{\varphi_n}(s)g(z)(\varphi_n(s,z)-1),u^{\varphi_n}(s)\rangle\nu(dz)ds\\
&\,\,\,\,\,\,-\int_0^T\int_\mathcal{Z} e^{-r(s)}\langle u^{\varphi_n}(s)g(z)(\varphi_n(s,z)-1),\tilde{u}_1^\varphi(s)\rangle\nu(dz)ds\bigg|\\
&\leq\int_0^T\int_\mathcal{Z} e^{-r(s)}\| u^{\varphi_n}(s)\||g(z)||\varphi_n(s,z)-1|\|u^{\varphi_n}(s)-\tilde{u}_1^\varphi(s)\|\nu(dz)ds\\
&\leq\sup_{0\leq s\leq T}\|u^{\varphi_n}(s)\|\int_0^T\int_\mathcal{Z} e^{-r(s)}\|u^{\varphi_n}(s)-\tilde{u}_1^\varphi(s)\||g(z)||\varphi_n(s,z)-1|\nu(dz)ds.
\end{align*}
By using the same technique in the proof of (\ref{X-Y-2}) and (\ref{normal_0}), we get (\ref{3344}).  As mentioned before, by   (\ref{X-Y-4}) and (\ref{X-Y-3}), we obtain $\tilde{u}_1^\varphi=u^\varphi$.

Next, we  prove that $u^{\varphi_n}\rightarrow u^\varphi$ in $C([0,T]; H)\cap L^{2\sigma+2}([0,T], L^{2\sigma+2}(D))$.
Denote $X_n:=u^{\varphi_n}-u^\varphi$. Then applying the chain rule to $\|X_n(s)\|^2$ and taking the real part, we obtain
 \begin{align*}
\frac{1}{2}\frac{d\|X_n(s)\|^2}{ds}
&=\,\textmd{Re}\langle G(u^{\varphi_n}(s))-G(u^\varphi(s)),X_n(s)\rangle\\
&\,\,\,\,\,\,+\,\textmd{Re}\int_{\mathcal{Z}}\langle u^{\varphi_n}(s)g(z)(\varphi_n(s,z)-1)-u^\varphi(s)g(z)(\varphi(s,z)-1),X_n(s)\rangle\\
&=\,\textmd{Re}\langle(1+i\alpha)(\Delta X_n(s),X_n(s)\rangle\\
&\,\,\,\,\,\,-\,\textmd{Re}\langle(1-i\beta)(|u^{\varphi_n}(s)|^{2\sigma}u^{\varphi_n}(s)-| u^\varphi(s))|^{2\sigma}u^\varphi(s)),X_n(s)\rangle\\
&\,\,\,\,\,\,+\gamma\,\textmd{Re}\langle X_n(s),X_n(s)\big>+\textmd{Re}\big<X_n(s),F(u^{\varphi_n}(s))-F(u^{\varphi}(s))\rangle\\
&\,\,\,\,\,\,+\,\textmd{Re}\int_{\mathcal{Z}}\langle X_n(s)g(z)(\varphi(s,z)-1),X_n(s)\rangle\nu(dz)\\
&\,\,\,\,\,\,+\,\textmd{Re}\int_{\mathcal{Z}}\langle u^{\varphi_n}(s)(g(z)(\varphi_n(s,z)-1)-g(z)(\varphi(s,z)-1)),X_n(s)\rangle\nu(dz)\\
&=:I^n_1+I^n_2+I^n_3+I^n_4+I^n_5+I^n_6.
\end{align*}

By calculating directly, we obtain $I^n_1(s)
=-\|\nabla X_n(s)\|^2$, $I^n_3(s)=\gamma\|X_n(s)\|^2$ and
\begin{align*}
I^n_5(s)
&\leq\int_{\mathcal{Z}}\|X_n(s)\|^2|g(z)||\varphi(s,z)-1|\nu(dz)\\
&=\|X_n(s)\|^2\int_{\mathcal{Z}}|g(z)||\varphi(s,z)-1|\nu(dz).
\end{align*}

  If $0<|\beta|<\frac{\sqrt{2\sigma+1}}{\sigma}$,  then by Lemma 3.5 of \cite{LG} we have
\begin{eqnarray*}
I^n_2(s)
\leq-\left(1-\frac{\sigma}{\sqrt{2\sigma+1}}|\beta|\right)2^{-2\sigma}\|X_n(s)\|^{2\sigma+2}_{2\sigma+2}<0.
\end{eqnarray*}

By using Lemma 3.6 in \cite{LG},   there exist   small enough positive parameters $\hat{\e}_1,\hat{\e}_2$ such that
\begin{equation}\label{Const4-4}
\begin{split}
I^n_4(s) \leq&\,  2\hat{\e}_1\|\nabla X_n(s)\|^2+2\hat{\e}_2\|X_n(s)\|^{2\sigma+2}_{2\sigma+2}\\
&+\Big(C_{4,4}(\e_5)+\e_6\|u^\varphi(s)\|^2+C_{4,5}(\e_7)\|\nabla u^\varphi(s)\|^{\frac{2\sigma}{\sigma-1}}\\
&+C_{4,6}(\e_8)\|\nabla u^\varphi(s)\|^{\frac{7\sigma+2}{\sigma+1}}+C_{4,7}(\e_9)\|\nabla u^\varphi(s)\|^{\frac{10\sigma+4}{\sigma+4}}\Big)\|X_n(s)\|^2\\
\leq &\, 2\hat{\e}_1\|\nabla X_n(s)\|^2+2\hat{\e}_2\|X_n(s)\|^{2\sigma+2}_{2\sigma+2}+C_{4,8}(\|\nabla u(0)\|^p+1)\|X_n(s)\|^2,
\end{split}
\end{equation}
where $2<\frac{7\sigma+2}{\sigma+1}<2\sigma$, $2<\frac{10\sigma+4}{\sigma+4}<2\sigma$, $\e_i\in(0,\infty)$,  $i=5,\ldots,9$ and $C_{4,j}\in(0,\infty)$, $j=4,\ldots,8$.  By using the method used in the proof of (\ref{X-Y-2}), we have
\begin{align*}
 \int^T_0|I^n_6(s)|ds
 \leq&\, \int^T_0\int_{\mathcal{Z}}\|u^{\varphi_n}(s)\||g(z)||\varphi_n(s,z)-1|\|X_n(s)\|\nu(dz)ds\\
&\,\,\,\,\,\,+\int^T_0\int_{\mathcal{Z}}\|u^{\varphi_n}(s)\||g(z)||\varphi(s,z)-1|\|X_n(s)\|\nu(dz)ds\\
 \leq&\, \sup_{s\in[0,T]}\|u^{\varphi_n}(s)\|\int^T_0\int_{\mathcal{Z}}|g(z)||\varphi_n(s,z)-1|\|X_n(s)\|\nu(dz)ds\\
&\,\,\,\,\,\,+\sup_{s\in[0,T]}\|u^{\varphi_n}(s)\|\int^T_0\int_{\mathcal{Z}}|g(z)||\varphi(s,z)-1|\|X_n(s)\|\nu(dz)ds\\
 \leq&\, \sup_{s\in[0,T]}\|u^{\varphi_n}(s)\|\bigg[2M\int_{A_{n,\e_4}}\int_{\mathcal{Z}}|g(z)||\varphi_n(s,z)-1|\nu(dz)ds\\
&\,\,\,\,\,\,+\e_4\int_{A^c_{n,\e_4}}\int_{\mathcal{Z}}|g(z)||\varphi_n(s,z)-1|\|\nu(dz)ds\bigg]\\
&\,\,\,\,\,\,+\sup_{s\in[0,T]}\|u^{\varphi_n}(s)\|\bigg[2M\int_{A_{n,\e_4}}\int_{\mathcal{Z}}|g(z)||\varphi(s,z)-1|\nu(dz)ds\\
&\,\,\,\,\,\,+\e_4\int_{A^c_{n,\e_4}}\int_{\mathcal{Z}}|g(z)||\varphi(s,z)-1|\|\nu(dz)ds\bigg].
\end{align*}
Since $\e_4$ is arbitrary and (\ref{4444}) holds, by Lemma \ref{3.1},  we obtain $$\int^T_0|I^n_6(s)|ds\rightarrow0 \ \ \ \ \text{as }n\rightarrow\infty.$$
Therefore, we have
\begin{align*}
&\frac{1}{2}\frac{d\|X_n(s)\|^2}{ds}
+\|\nabla X_n(s)\|^2
+\left(1-\frac{\sigma}{\sqrt{2\sigma+1}}|\beta|\right)2^{-2\sigma}\|X_n(s)\|^{2\sigma+2}_{2\sigma+2}\\
 \leq&\,\gamma\|X_n(s)\|^2+I^n_4(s)+\|X_n(s)\|^2\int_{\mathcal{Z}}|g(z)||\varphi(s,z)-1|\nu(dz)
+I^n_6(s).
\end{align*}
By Gronwall's inequality and (\ref{g_H_2}) in Lemma \ref{3.1}, as $n\rightarrow\infty$,
\begin{align*}
&\sup_{t\in[0,T]}\|X_n(t)\|^2+\int^T_0\|\nabla X_n(s)\|^2ds+\int^T_0\left(1-\frac{\sigma}{\sqrt{2\sigma+1}}|\beta|\right)2^{-2\sigma}\| X_n(s)\|^{2\sigma+2}_{2\sigma+2}ds\\
 \leq&\, \exp\left[\int^T_0\left(2\gamma+C_{4,8}(\|\nabla u(0)\|^p+1)+2\int_{\mathcal{Z}}|g(z)||\varphi(s,z)-1|\nu(dz)\right)ds\right]\int^T_0|I^n_6(s)|ds\\
 \leq &\, e^{\left(2\gamma T+C_{4,8}(\|\nabla u(0)\|^p+1)T+2C^0_2\right)}\int^T_0|I^n_6(s)|ds\rightarrow0.
\end{align*}
The proof is complete.
\end{proof}

\section{ Verification of Condition \ref{LDP} (b) }
   Recall   $\tilde{\mathbb{A}}^N, N\ge1$ in Subsection 2.2. For any $\varphi_\e \in \tilde{\mathbb{A}}^N$, consider the following controlled SPDE
\begin{equation}\label{eq SPDE 02}
\begin{split}
d \tilde{u}^\e (t)
&=A(\tilde{u}^\e (t))dt
+B(\tilde{u}^\e (t))dt+\int_{\ZZ}g(z)\tilde{u}^\e (t)(\varphi_\e (t,z)-1)\nu(dz)dt \\
&\,\,\,\,\,\,+\e \int_{\ZZ}g(z)\tilde{u}^\e (t-)\tilde{\eta}^{\e ^{-1}\varphi_\e }(dz,dt),
\end{split}
\end{equation}
with the initial condition $\widetilde{u}^\e (0)=u(0)$.

Let $\nu_{\e}=\frac{1}{\varphi_{\e}}$. The following
lemma is taken  from \cite[Lemma 2.3]{Budhiraja-Dupuis-Maroulas.}.
\begin{lemma}Let
\begin{eqnarray*}
\mathcal{E}^{\e}_{t}(\nu_{\e})&:=&\exp\Big\{\int_{(0,t]\times \ZZ\times
[0,\e^{-1}\varphi_{\e}]}\log(\nu_{\e}(s,x))
\bar{\eta}(ds\;dx\;dr)\\
&&\qquad +\int_{(0,t]\times \ZZ\times
[0,\e^{-1}\varphi_{\e}]}(-\nu_{\e}(s,x)+1)\bar{\nu}_{T}(ds\;dx\;dr)\Big\}.
\end{eqnarray*}
Then
$$
\mathbb{Q}^{\e}_{t}(G)=\int_{G}\mathcal{E}^{\e}_{t}(\vartheta_{\e})\;d\bar{\mathbb{P}},
\quad
\text{for}\quad G\in \mathcal{B}(\bar{\mathbb{M}})
$$
defines a probability measure on $\bar{\mathbb{M}}$.
\end{lemma}
Since $\e
\eta^{\e^{-1}\varphi_{\e}}$ under $\mathbb{Q}^{\e}_{T}$ has the same law as
that of $\e \eta^{\e^{-1}}$ under $\bar{\mathbb{P}}$,  $\tilde{u}^\e=\mathcal{G}^{\e}\big(\e \eta^{\e ^{-1}\varphi_\e}\big)$ solves following controlled stochastic generalized Ginzburg-Landau equation:
\begin{equation}\label{eq SPDE 02}
\begin{split}
 \tilde{u}^\e (t)
&=u(0)+\int^t_0(1+i\alpha)\Delta\tilde{u}^\e (s)ds-\int^t_0(1-i\beta)|\tilde{u}^\e(s)|^{2\sigma} \tilde{u}^\e(s)ds\\
&\,\,\,\,\,+\gamma\int^t_0\tilde{u}^\e(s)ds
+\int^t_0\left(\lambda_1\cdot \nabla (|\tilde{u}^\e(s)|^2\tilde{u}^\e(s))+(\lambda_2\cdot \nabla \tilde{u}^\e(s))|\tilde{u}^\e(s)|^2\right)ds\\
&\,\,\,\,\,+\int^t_0\int_{\ZZ}g(z)\tilde{u}^\e (s-)
\left(\epsilon \eta^{\e ^{-1}\varphi_\e }(dz,ds)-\nu(dz)ds\right)\\
&=u(0)+\int^t_0(1+i\alpha)\Delta\tilde{u}^\e (s)ds-\int^t_0(1-i\beta)|\tilde{u}^\e(s)|^{2\sigma} \tilde{u}^\e(s)ds\\
&\,\,\,\,\,+\gamma\int^t_0\tilde{u}^\e(s)ds
+\int^t_0\int_{\ZZ}g(z)\tilde{u}^\e (s)(\varphi_\e (s,z)-1)\nu(dz)ds\\
&\,\,\,\,\,+\int^t_0\left(\lambda_1\cdot \nabla (|\tilde{u}^\e(s)|^2\tilde{u}^\e(s))+(\lambda_2\cdot \nabla \tilde{u}^\e(s))|\tilde{u}^\e(s)|^2\right)ds\\
&\,\,\,\,\,+\int^t_0\int_{\ZZ}\epsilon g(z)\tilde{u}^\e (s-)\left(\eta^{\e ^{-1}\varphi_\e }(dz,ds)-\e ^{-1}\varphi_\e (s,z)\nu(dz)ds\right).
\end{split}
\end{equation}

\begin{lemma}\label{Lemma4-2}
Assume $u(0)\in V$ is deterministic. There exists $\e_0>0$ such that
\begin{eqnarray}\label{eq SPDE 10}
\sup_{0<\e<\e_0} \mathbb{E}\left[\sup_{0<t<T}\|\tilde{u}^\e (t)\|^2
+\int^T_0\|\nabla\tilde{u}^\e (t)\|^2dt+\int^T_0\|\tilde{u}^\e (t)\|^{2\sigma+2}_{2\sigma+2}dt\right]<\infty.
\end{eqnarray}
\end{lemma}

\begin{proof}
Applying the It\^{o} formula to $\|\tilde{u}^\e(t)\|^2$ and taking the real part, by (\ref{eq SPDE 02}), we have
\begin{align*}
\|\tilde{u}^\e(t)\|^2
&=\|u(0)\|^2+ 2\textmd{Re}\int^t_0\langle(1+i\alpha)\Delta\tilde{u}^\e(s),\tilde{u}^\e(s)\rangle ds\\
&\,\,\,\,\,-2\textmd{Re}\int^t_0\langle(1-i\beta)|\tilde{u}^\e(s)|^{2\sigma}\tilde{u}^\e(s),\tilde{u}^\e(s)\rangle ds
+2\gamma\int^t_0\langle\tilde{u}^\e(s),\tilde{u}^\e(s)\rangle ds\\
&\,\,\,\,\,+2\textmd{Re}\int^t_0\langle\left(\lambda_1\cdot \nabla (|\tilde{u}^\e(s)|^2\tilde{u}^\e(s))+(\lambda_2\cdot \nabla \tilde{u}^\e(s))|\tilde{u}^\e(s)|^2\right),\tilde{u}^\e(s)\rangle ds\\
&\,\,\,\,\,+2\int^t_0\int_{\ZZ}\langle\tilde{u}^\e(s)g(z)(\varphi_\varepsilon(s,z)-1),\tilde{u}^{\e}(s)\rangle\nu(dz)ds\\
&\,\,\,\,\,+2\int^t_0\int_{\ZZ}\langle\e g(z)\tilde{u}^\e(s-),\tilde{u}^\e(s)\rangle\left(\eta^{\e^{-1}\varphi_\e}(dz,ds)-\e^{-1}\varphi_\e(s,z)\nu(dz)ds\right)\\
&\,\,\,\,\,+\int^t_0\int_{\ZZ}\|\e g(z)\tilde{u}^\e(s-)\|^2\left(\eta^{\e^{-1}\varphi_\e}(dz,ds)-\e^{-1}\varphi_\e(s,z)\nu(dz)ds\right)\\
&\,\,\,\,\,+\e\int^t_0\int_{\ZZ}|g(z)|^2\|\tilde{u}^\e(s-)\|^2\varphi_\varepsilon(s,z)\nu(dz)ds\\
&=:\|u(0)\|^2-\int^t_0\|\nabla\tilde{u}^\e(s)\|^2ds-2\int^t_0\|\tilde{u}^\e(s)\|^{2\sigma+2}_{2\sigma+2}ds+2\gamma\int^t_0\|\tilde{u}^\e(s)\|^2ds\\
&\,\,\,\,\,\,+I^{\e}_1(t)+I^{\e}_2(t)+I^{\e}_3(t)+I^{\e}_4(t)+I^{\e}_5(t).
\end{align*}

For  $I^{\e}_1$ and $I^{\e}_2$,  we have
\begin{align}\label{eq:eqI1new}
I^{\e}_1(t)
\leq\frac{1}{2}\int^t_0\|\nabla \tilde{u}^\e(s)\|^2ds+\frac{1}{2}\int^t_0\|\tilde{u}^\e(s)\|^{2\sigma+2}_{2\sigma+2}ds+C_{5,1}(|\lambda_1|,|\lambda_2|)
\int^t_0\|\tilde{u}^\e(s)\|^2ds
\end{align}
and
\begin{align}\label{eq:eqI2}
I^{\e}_2(t)
  \leq2\int^t_0\int_{\ZZ}\|\tilde{u}^\e(s)\|^2 |g(z)||\varphi_\varepsilon(s,z)-1|\nu(dz)ds.
\end{align}

For  $I^{\e}_3(t)$, applying Burkholder-Davis-Gundy's inequality and (\ref{g_H_1}) in Lemma \ref{3.1}, we have
\begin{equation}\label{eq:eqI3}
\begin{split}
&\, \mathbb{E}\left(\sup_{t\in[0,T]}|I^{\e}_3(t)|\right)\\
 \leq&\,  4\mathbb{E}\left(\int^t_0\int_{\ZZ}\big|2\big<\e \tilde{u}^\e(s-)g(z),\tilde{u}^\e(s)\big>\big|^2\eta^{\e^{-1}\varphi_\e}(dz,ds)\right)^{\frac{1}{2}}\\
 \leq &\,  8\mathbb{E}\left(\int^t_0\int_{\ZZ}\e^2\|\tilde{u}^\e(s-)\|^2\|\tilde{u}^\e(s)\|^2| g(z)|^2\eta^{\e^{-1}\varphi_\e}(dz,ds)\right)^{\frac{1}{2}}\\
 \leq&\,  8\mathbb{E}\left(\sup_{t\in[0,T]}\|\tilde{u}^\e(t)\|\left(\int^t_0\int_{\ZZ}\e^2\|\tilde{u}^\e(s-)\|^2| g(z)|^2\eta^{\e^{-1}\varphi_\e}(dz,ds)\right)^{\frac{1}{2}}\right)\\
 \leq &\, \frac{1}{C_{5,2}}\mathbb{E}\left(\sup_{t\in[0,T]}\|\tilde{u}^\e(t)\|^2\right)
+16C_{5,2}\e^2\mathbb{E}\left(\int^t_0\int_{\ZZ}\|\tilde{u}^\e(s-)\|^2|g(z)|^2\eta^{\e^{-1}\varphi_\e}(dz,ds)\right)\\
 =&\,  \frac{1}{C_{5,2}}\mathbb{E}\left(\sup_{t\in[0,T]}\|\tilde{u}^\e(t)\|^2\right)
+16C_{5,2}\e\mathbb{E}\left(\int^t_0\int_{\ZZ}\|\tilde{u}^\e(s-)\|^2|g(z)|^2\varphi_\e(s,z)\nu(dz)ds\right)\\
   \leq&\,  \left(\frac{1}{C_{5,2}}+16\e C_{5,2}\cdot C^0_1\right)\mathbb{E}\left(\sup_{t\in[0,T]}\|\tilde{u}^\e(t)\|^2\right),
\end{split}
\end{equation}
where $C_{5,2}\in (0,\infty)$.

For the fourth term  $I^{\e}_4(t)$, we have
\begin{equation}\label{eq:eqI4}
\begin{split}
\mathbb{E}\left(\sup_{t\in[0,T]}|I^{\e}_4(t)|\right)
 \leq&\, \mathbb{E}\left(\int^t_0\int_{\ZZ}\|\e\tilde{u}^\e(s-)\|^2|g(z)|^2\eta^{\e^{-1}\varphi_\e}(dz,ds)\right)\\
&\,\,\,\,\,\,+\e\mathbb{E}\left(\int^t_0\int_{\ZZ}\|\tilde{u}^\e(s)\|^2|g(z)|^2\varphi_\varepsilon(s,z)\nu(dz)ds\right)\\
 \leq& \, 2\e\mathbb{E}\left(\int^t_0\int_{\ZZ}\|\tilde{u}^\e(s)\|^2|g(z)|^2\varphi_\varepsilon(s,z)\nu(dz)ds\right)\\
 \leq& \, 2\e C^0_1\mathbb{E}\left(\sup_{t\in[0,T]}\|\tilde{u}^\e(t)\|^2\right).
\end{split}
\end{equation}
For $I^{\e}_5(t)$, we have
\begin{eqnarray}\label{eq:eqI5}
I^{\e}_5(t)=\e\int^t_0\|\tilde{u}^\e(s-)\|^2\int_{\ZZ}|g(z)|^2\varphi_\varepsilon(s,z)\nu(dz)ds.
\end{eqnarray}
By the above inequalities and  Gronwall's inequality, we have
\begin{align*}
&\,\,\,\,\,\|\tilde{u}^\e (t)\|^2
+\frac{3}{2}\int^t_0\|\nabla\tilde{u}^\e (s)\|^2ds+\frac{3}{2}\int^t_0\|\tilde{u}^\e (s)\|^{2\sigma+2}_{2\sigma+2}ds\\
&\leq\left(\|\tilde{u}^\e (0)\|^2+\sup_{t\in[0,T]}|I^{\e}_3(t)|+\sup_{t\in[0,T]}|I^{\e}_4(t)|\right)\\
&\,\,\,\,\,\cdot\exp \bigg((2\gamma+C_{5,1}(|\lambda_1|,|\lambda_2|))T+2\int^t_0\int_{\ZZ}|g(z)||\varphi_\varepsilon(s,z)-1|\nu(dz)ds\\
&\,\,\,\,\,+\e\int^t_0\int_{\ZZ}|g(z)|^2\varphi_\varepsilon(s,z)\nu(dz)ds\bigg)\\
&\leq C_{5,3}\left(\|\tilde{u}^\e (0)\|^2+\sup_{t\in[0,T]}|I^{\e}_3(t)|+\sup_{t\in[0,T]}|I^{\e}_4(t)|\right),
\end{align*}
where $C_{5,3}= e^{\left((2\gamma+C_{5,1}(|\lambda_1|,|\lambda_2|))T+2C^0_2+\e C^0_1\right)}$.
Therefore,  by \eqref{eq:eqI3} and \eqref{eq:eqI4}, we have
\begin{equation}\label{Control-u}
\begin{split}
&\,\,\,\,\,\,\mathbb{E}\left(\sup_{0<t<T}\|\tilde{u}^\e (t)\|^2
+\int^T_0\|\nabla\tilde{u}^\e (s)\|^2ds+\int^T_0\|\tilde{u}^\e (s)\|^{2\sigma+2}_{2\sigma+2}ds\right)\\
\leq &\, C_{5,3}\mathbb{E}\|\tilde{u}^\e (0)\|^2+C_{5,3}\left(2\e C^0_1+\frac{1}{C_{5,2}}+16\e C_{5,2}\cdot C^0_1\right)\cdot\mathbb{E}\sup_{0<t<T}\|\tilde{u}^\e (t)\|^2.
\end{split}
\end{equation}
Choosing constant $C_{5,2}$ large enough and $\e_0>0$ small enough  such that
 \begin{eqnarray*}
C_{5,3}\left(2\e C^0_1+\frac{1}{C_{5,2}}+16\e C_{5,2}\cdot C^0_1\right)<\frac{1}{2},
\end{eqnarray*}
 (\ref{Control-u})  implies   (\ref{eq SPDE 10}).
The proof is complete.
\end{proof}

According to  Proposition \ref{Main-thm-01},  we know that for any  $\varphi_{\e} \in \mathbb{S}$,  there exists a unique solution $u^{\varphi_\e}(t)$ to the following equation:
\begin{align}\label{P-u-1}
\nonumber
u^{\varphi_\e}(t)&=u^{\varphi_\e}(0)+ \int^t_0(1+i\alpha)\Delta u^{\varphi_\e}(s)ds\\
\nonumber
&\,\,\,\,\,+\int^t_0\big(-(1-i\beta)|u^{\varphi_\e}(s)|^{2\sigma}u^{\varphi_\e}(s)+\gamma u^{\varphi_\e}(s))  \big)ds\\
&\,\,\,\,\,+\int^t_0\lambda_1\cdot \nabla (|u^{\varphi_\e}(s)|^2u^{\varphi_\e}(s))+(\lambda_2\cdot \nabla u^{\varphi_\e}(s))|u^{\varphi_\e}(s)|^2ds\\
\nonumber
&\,\,\,\,\,+\int^T_0\int_\mathcal{Z}g(z)(\varphi_\e(s,z)-1)u^{\varphi_\e}(s)\nu(dz)ds,
\end{align}
with the initial value $u^{\varphi_\e}(0)=u(0)$.
 By the definition of $\mathcal{G}^{\e}$, we have
\begin{eqnarray}\label{define G-epsilon}
\tilde{u}^{\e}=\mathcal{G}^{\e}\Big(\e
\eta^{\e^{-1}\varphi_{\e}}\Big).
\end{eqnarray}

The following proposition, together with Chebychev's inequality,  implies Condition \ref{LDP}(B). 
\begin{proposition}\label{Prop-2}
For $\e>0$, let $\varphi_\e\in\tilde{\mathbb{A}}^N$. Denote that $\mathcal{G}^0(\varphi_\e)=u^{\varphi_\e}$. Then

\begin{eqnarray*}
\lim_{\e\rightarrow0}\mathbb{E}\left(\sup_{t\in[0,T]}\|\tilde{u}^\e (t)-u^{\varphi_\e}(t)\|^2
+\int^T_0\|\tilde{u}^\e (s)-u^{\varphi_\e}(s)\|^2_Vds+\int^T_0\|\tilde{u}^\e (s)-u^{\varphi_\e}(s)\|^{2\sigma+2}_{2\sigma+2}ds\right)=0.
\end{eqnarray*}
\end{proposition}

\begin{proof} Note that
\begin{align*}
\tilde{u}^\e (t)-u^{\varphi_\e}(t)
 =&\, \int^t_0(1+i\alpha)\Delta \left(\tilde{u}^\e (s)- u^{\varphi_\e}(s)\right)ds\\
&\,\,\,\,\,-\int^t_0(1-i\beta)\left(|\tilde{u}^\e (s)|^{2\sigma}\tilde{u}^\e (s)-|u^{\varphi_\e}(s)|^{2\sigma}u^{\varphi_\e}(s)\right)ds\\
&\,\,\,\,\,+\int^t_0\gamma\left(\tilde{u}^\e (s)- u^{\varphi_\e}(s)\right) ds
+\int^t_0\left(F(\tilde{u}^\e (s))-F(u^{\varphi_\e}(s))\right)ds\\
&\,\,\,\,\,+\int^t_0\int_{\ZZ}\left(\tilde{u}^\e (s)- u^{\varphi_\e}(s)\right)g(z)(\varphi_\e(s,z)-1)\nu(dz)ds\\
&\,\,\,\,\,+\int^t_0\int_{\ZZ}\epsilon \tilde{u}^\e (s-)g(z)\left(\eta^{\e ^{-1}\varphi_\e }(dz,ds)-\e ^{-1}\varphi_\e (s,z)\nu(dz)ds\right).\\
\end{align*}

Applying the It\^o formula to $\|\tilde{u}^\e (t)- u^{\varphi_\e}(t)\|^2$ and taking the real part, we have
\begin{align*}
&\,\,\,\,\,\|\tilde{u}^\e (t)- u^{\varphi_\e}(t)\|^2\\
&=2\textmd{Re}\int^t_0\langle\tilde{u}^\e (s)- u^{\varphi_\e}(s),(1+i\alpha)\Delta \left(\tilde{u}^\e (s)- u^{\varphi_\e}(s)\right)\rangle ds\\
&\,\,\,\,\,-2\textmd{Re}\int^t_0\langle\tilde{u}^\e (s)- u^{\varphi_\e}(s),(1-i\beta)\left(|\tilde{u}^\e (s)|^{2\sigma}\tilde{u}^\e (s)-|u^{\varphi_\e}(s)|^{2\sigma}u^{\varphi_\e}(s)\right)\rangle ds\\
&\,\,\,\,\,+2\gamma \textmd{Re}\int^t_0\langle\tilde{u}^\e (s)- u^{\varphi_\e}(s),\tilde{u}^\e (s)- u^{\varphi_\e}(s)\rangle ds\\
&\,\,\,\,\,+2\textmd{Re}\int^t_0\langle\tilde{u}^\e (s)- u^{\varphi_\e}(s),F(\tilde{u}^\e (s))-F(u^{\varphi_\e}(s))\rangle ds\\
&\,\,\,\,\,+2\textmd{Re}\int^t_0\int_{\ZZ}\langle\tilde{u}^\e (s)- u^{\varphi_\e}(s),\left(\tilde{u}^\e (s)- u^{\varphi_\e}(s)\right)g(z)(\varphi_\e(s,z)-1)\rangle\nu(dz)ds\\
&\,\,\,\,\,+\textmd{Re}\int^t_0\int_{\ZZ}2\langle\e g(z)\tilde{u}^\e(s-),\tilde{u}^\e (s)- u^{\varphi_\e}(s)\rangle\left(\eta^{\e^{-1}\varphi_\e}(dz,ds)-\e^{-1}\varphi_\e(s,z)\nu(dz)ds\right)\\
&\,\,\,\,\,+\int^t_0\int_{\ZZ}\|\e g(z)\tilde{u}^\e(s-)\|^2\left(\eta^{\e^{-1}\varphi_\e}(dz,ds)-\e^{-1}\varphi_\e(s,z)\nu(dz)ds\right)\\
&\,\,\,\,\,+\e\int^t_0\int_{\ZZ}|g(z)|^2\|\tilde{u}^\e(s-)\|^2\varphi_\epsilon(s,z)\nu(dz)ds\\
&=:J^{\e}_1(t)+J^{\e}_2(t)+J^{\e}_3(t)+J^{\e}_4(t)+J^{\e}_5(t)+J^{\e}_6(t)+J^{\e}_7(t)+J^{\e}_8(t).
\end{align*}
A straightforward calculation shows
\begin{equation}\label{J-1-est}
 \begin{split}
J^{\e}_1(t)&=-2\textmd{Re}\int^t_0\langle\nabla\left(\tilde{u}^\e (s)- u^{\varphi_\e}(s)\right),(1+i\alpha)\nabla \left(\tilde{u}^\e (s)- u^{\varphi_\e}(s)\right)\rangle ds\\
&=-2\int^t_0\|\nabla\left(\tilde{u}^\e (s)- u^{\varphi_\e}(s)\right)\|^2ds;
\end{split}
\end{equation}
\begin{align}\label{J-3-est}
J^{\e}_3(t)=2\gamma\int^t_0\|\tilde{u}^\e (s)- u^{\varphi_\e}(s)\|^2ds;
\end{align}
and
\begin{equation}\label{J-5-est}
\begin{split}
J^{\e}_5(t) \leq&\, 2\int^t_0\int_{\ZZ}\|\tilde{u}^\e (s)- u^{\varphi_\e}(s)\|^2|g(z)||\varphi_\e(s,z)-1|\nu(dz)ds\\
=&\, 2\int^t_0\|\tilde{u}^\e (s)- u^{\varphi_\e}(s)\|^2\int_{\ZZ}|g(z)||\varphi_\e(s,z)-1|\nu(dz)ds.
\end{split}
\end{equation}

According to Lemma 3.5 in Lin and Gao \cite{LG}, for $0<|\beta|<\frac{\sqrt{2\sigma+1}}{\sigma}$, we have
\begin{align*}
&\,\,\,\,\,\textmd{Re}\left[-(1-i\beta)\langle\tilde{u}^\e (s)- u^{\varphi_\e}(s),\left(|\tilde{u}^\e (s)|^{2\sigma}\tilde{u}^\e (s)-|u^{\varphi_\e}(s)|^{2\sigma}u^{\varphi_\e}(s)\right)\rangle\right]\\
&\leq-\left(1-\frac{\sigma}{\sqrt{2\sigma+1}}|\beta|\right)2^{-2\sigma}\|\tilde{u}^\e (s)- u^{\varphi_\e}(s)\|^{2\sigma+2}_{2\sigma+2}.
\end{align*}
Therefore,
\begin{align}\label{J-2-est}
J^{\e}_2(t)\leq-\left(1-\frac{\sigma}{\sqrt{2\sigma+1}}|\beta|\right)2^{-2\sigma}\int^t_0\|\tilde{u}^\e (s)- u^{\varphi_\e}(s)\|^{2\sigma+2}_{2\sigma+2}ds<0.
\end{align}

By using Lemma 3.6 in \cite{LG}, we have that there exist small parameters $\hat{\e}_1,\hat{\e}_2$ such that
\begin{align}\label{J-4-est}
\nonumber
J^{\e}_4(t)&\leq 2\hat{\e}_1\int^t_0\|\nabla(\tilde{u}^\e (s)- u^{\varphi_\e}(s))\|^2ds+2\hat{\e}_2\int^t_0|\tilde{u}^\e (s)- u^{\varphi_\e}(s)\|^{2\sigma+2}_{2\sigma+2}ds\\
&+2\int^t_0\big(C_{4,4}(\e_5)+\e_6\|u^{\varphi_\e}(s)\|^2+C_{4,5}(\e_7)\|\nabla u^{\varphi_\e}(s)\|^{\frac{2\sigma}{\sigma-1}}\\
\nonumber
&+C_{4,6}(\e_8)\|\nabla u^{\varphi_\e}(s)\|^{\frac{7\sigma+2}{\sigma+1}}+C_{4,7}(\e_9)\|\nabla u^{\varphi_\e}(s)\|^{\frac{10\sigma+4}{\sigma+4}}\big)\|\tilde{u}^\e (s)- u^{\varphi_\e}(s)\|^2ds\\
\nonumber
&\leq2C_{4,8}(\|\nabla u(0)\|^p+1)\int^t_0\|\tilde{u}^\e (s)- u^{\varphi_\e}(s)\|^2ds,
\end{align}
  where $2<\frac{7\sigma+2}{\sigma+1}<2\sigma$, $2<\frac{10\sigma+4}{\sigma+4}<2\sigma$, $\e_i$ $(i=5,\dots,9)$ and $C_{4,j}$ $(j=4,\dots, 8)$  are   constants  appeared in (\ref{Const4-4}). 

Combining (\ref{J-1-est}),  (\ref{J-3-est}), (\ref{J-5-est}), (\ref{J-2-est}) and  (\ref{J-4-est}) and  applying Gronwall's inequality, we have
\begin{equation}\label{644}
\begin{split}
&\,\,\,\,\,\sup_{t\in[0,T]}\|\tilde{u}^\e (t)- u^{\varphi_\e}(t)\|^2+\int^T_0\|\tilde{u}^\e (t)-u^{\varphi_\e}(t)\|^2_Vds+\int^T_0\|\tilde{u}^\e (s)-u^{\varphi_\e}(s)\|^{2\sigma+2}_{2\sigma+2}ds\\
&\leq\left[\sup_{t\in[0,T]}|J^{\e}_6(t)|+\sup_{t\in[0,T]}|J^{\e}_7(t)|+\sup_{t\in[0,T]}|J^{\e}_8(t)|\right]\\
&\,\,\,\,\,\cdot\exp{\left[\int^T_0\left(2\gamma+2C_{4,8}(\|\nabla u(0)\|^p+1)+\int_{\ZZ}|g(z)||\varphi_\e(s,z)-1|\nu(dz)\right)ds\right]}\\
&\leq\left[\sup_{t\in[0,T]}|J^{\e}_6(t)|+\sup_{t\in[0,T]}|J^{\e}_7(t)|+\sup_{t\in[0,T]}|J^{\e}_8(t)|\right]
\cdot e^{\left(2\gamma T+2C_{4,8}(\|\nabla u(0)\|^p+1)T+C^0_2\right)}.
\end{split}
\end{equation}
Similar to $(\ref{eq:eqI3})$, $(\ref{eq:eqI4})$ and $(\ref{eq:eqI5})$, we have{\small
\begin{equation}\label{J-6-est}
\begin{split}
  & \mathbb{E}\left(\sup_{t\in[0,T]}|J^{\e}_6(t)|\right)\\
   \leq&\, 4\mathbb{E}\left(\int^t_0\int_{\ZZ}|2\langle\e \tilde{u}^\e(s-)g(z),\widetilde{u}^\e (s)- u^{\varphi_\e}(s)\rangle|^2\eta^{\e^{-1}\varphi_\e}(dz,ds)\right)^{\frac{1}{2}}\\
 \leq &\, 8\mathbb{E}\left(\int^t_0\int_{\ZZ}\e^2\|\tilde{u}^\e(s-)\|^2| g(z)|^2\|\tilde{u}^\e (s)- u^{\varphi_\e}(s)\|^2\eta^{\e^{-1}\varphi_\e}(dz,ds)\right)^{\frac{1}{2}}\\
 \leq&\, 8\e\mathbb{E}\left(\sup_{t\in[0,T]}\|\tilde{u}^\e (s)- u^{\varphi_\e}(s)\|\left(\int^t_0\int_{\ZZ}\|\tilde{u}^\e(s-)\|^2| g(z)|^2\eta^{\e^{-1}\varphi_\e}(dz,ds)\right)^{\frac{1}{2}}\right)\\
 \leq &\, \frac{4\sqrt{\e}}{C_{5,2}}\mathbb{E}\left(\sup_{t\in[0,T]}\|\tilde{u}^\e (t)- u^{\varphi_\e}(t)\|^2\right)
+4\sqrt{\e}C_{5,2}\mathbb{E}\left(\int^t_0\int_{\ZZ}\e\|\tilde{u}^\e(s-)\|^2|g(z)|^2\eta^{\e^{-1}\varphi_\e}(dz,ds)\right)\\
 =&\, \frac{4\sqrt{\e}}{C_{5,2}}\mathbb{E}\left(\sup_{t\in[0,T]}\|\tilde{u}^\e (t)- u^{\varphi_\e}(t)\|^2\right)
+4\sqrt{\e}C_{5,2}\mathbb{E}\left(\int^t_0\int_{\ZZ}\|\tilde{u}^\e(s-)\|^2| g(z)|^2\varphi_\e(s,z)\nu(dz)ds\right)\\
 \leq&\, \frac{4\sqrt{\e}}{C_{5,2}}\mathbb{E}\left(\sup_{t\in[0,T]}\|\tilde{u}^\e (t)- u^{\varphi_\e}(t)\|^2\right)+4\sqrt{\e}C_{5,2}\cdot C^0_1\mathbb{E}\left(\sup_{t\in[0,T]}\|\tilde{u}^\e(t)\|^2\right);
\end{split}
\end{equation}}
\begin{align}\label{J-7-est}
\mathbb{E}\left(\sup_{t\in[0,T]}|J^{\e}_7(t)|\right)\leq
2\e C^0_1\mathbb{E}\left(\sup_{t\in[0,T]}\|\tilde{u}^\e(t)\|^2\right);
\end{align}
and
\begin{align}\label{J-8-est}
\nonumber
\mathbb{E}\left(\sup_{t\in[0,T]}|J^{\e}_8(t)|\right)&\leq
\e\mathbb{E}\left(\sup_{t\in[0,T]}\|\tilde{u}^\e(t)\|^2\int^t_0\int_{\ZZ}|g(z)|^2\varphi_\epsilon(s,z)\nu(dz)ds\right)\\
&\leq\e C^0_1\mathbb{E}\left(\sup_{t\in[0,T]}\|\tilde{u}^\e(t)\|^2\right).
\end{align}
In view of (\ref{644}), (\ref{J-6-est}), (\ref{J-7-est}) and (\ref{J-8-est}),  we know    that there exist constants $C_{5,4}\in (0,\infty)$ and $\tilde{\e}_0>0$ such that for any $0<\e\leq\tilde{\e}_0$,
\begin{align*}
&\mathbb{E}\left(\sup_{t\in[0,T]}\|\tilde{u}^\e (t)-u^{\varphi_\e}(t)\|^2
+\int^T_0\|\tilde{u}^\e (s)-u^{\varphi_\e}(t)\|^2_Vdt+\int^T_0\|\tilde{u}^\e (t)-u^{\varphi_\e}(t)\|^{2\sigma+2}_{2\sigma+2}dt\right)\\
&\leq
\sqrt{\e}C_{5,4}\cdot e^{\left(2\gamma T+2C_{4,8}(\|\nabla u(0)\|^p+1)T+C^0_2\right)},
\end{align*}
which completes the proof of  this proposition.
\end{proof}

 \vskip0.5cm
\noindent{\bf Acknowledgments}:  The research of  R. Wang  is partially supported by  NNSFC grant 11871382. The research of B. Zhang is
partially supported by NNSFC grants 11971361 and 11731012.

\end{document}